\RequirePackage{fix-cm}
\documentclass[smallextended]{svjour3}       
\smartqed  
\usepackage{graphicx}
 \usepackage{mathptmx}      
\usepackage{graphicx}
\usepackage{amssymb}
\usepackage{amsmath}
\usepackage{subcaption}
\usepackage{wrapfig}
\usepackage{color}
\usepackage{hyperref}
\usepackage{pinlabel}
\usepackage{tcolorbox}
\usepackage[alwaysadjust]{paralist}

\graphicspath{{./figs/}}
\newtheorem{theorem2}{Theorem}
\newtheorem{proposition2}[theorem2]{Proposition}
\newtheorem{lemma2}[theorem2]{Lemma}
\newtheorem{definition2}[theorem2]{Definition}
\newtheorem{corollary2}[theorem2]{Corollary}
\newtheorem{remark2}[theorem2]{Remark}
\newtheorem{fact}[theorem2]{Fact}
\newtheorem{convention}[theorem2]{Convention}
\spnewtheorem*{theorem*}{Theorem}{\bf}{\it}
\spnewtheorem*{lemma*}{Lemma}{\bf}{\it}
\spnewtheorem*{yamabe}{Yamabe Problem}{\bf}{\it}
\numberwithin{theorem2}{section}
\newcommand{\Z}{\mathbb{Z}}
\newcommand{\LobL}{\mathbb{L}}
\newcommand{\N}{\mathbb{N}}
\newcommand{\R}{\mathbb{R}}

\newcommand{\E}{\mathbb{E}}
\newcommand{\F}{\mathbb{F}}
\newcommand{\pr}[2] {\frac{\partial #1}{\partial #2}}

\newcommand{\abs}[1] {\vert #1\vert}
\newcommand{\1}{\mathbb{I}}

\definecolor{redg}{RGB}{204,0,0}

\journalname{Discrete \& Computational Geometry}
\begin{document}
\title{Discrete Yamabe problem for polyhedral surfaces\thanks{This research was supported by DFG SFB/Transregio~109 ``Discretization in Geometry and Dynamics" and the Wittgenstein Prize, Austrian Science Fund (FWF), grant no. Z 342-N31.}
}

\author{Hana Dal Poz Kou\v{r}imsk\'{a}       
}
\institute{H. Dal Poz Kou\v{r}imsk\'{a} \at
              Institute of Science and Technology Austria,\\
              Am Campus 1,
              3400 Klosterneuburg,
              Austria\\
              \email{hana.kourimska@ist.ac.at}           
}

\date{Received: date / Accepted: date}

\maketitle

\begin{abstract}
We study a new discretization of the Gaussian curvature for polyhedral surfaces. This discrete Gaussian curvature is defined on each conical singularity of a polyhedral surface as the quotient of the angle defect and the area of the Voronoi cell corresponding to the singularity.
\\
We divide polyhedral surfaces into discrete conformal classes using a generalization of discrete conformal equivalence pioneered by Feng Luo. We subsequently show that, in every discrete conformal class, there exists a polyhedral surface with constant discrete Gaussian curvature. We also provide explicit examples to demonstrate that this surface is in general not unique.
\end{abstract}

\section{Introduction}
The Yamabe problem asks if every closed Riemannian manifold is conformally equivalent to one with constant scalar curvature. More precisely:
\begin{yamabe}
	Let $g$ be a Riemannian metric on a closed smooth manifold~$M$. Does there exist a smooth function~$u$ on~$M$ such that the Riemannian metric~$e^{2u} g$ has constant scalar curvature?
\end{yamabe} 
For two-dimensional manifolds the scalar and the Gaussian curvature are equivalent, and thus the Yamabe problem is answered by the celebrated \emph{Poincar\'{e}-Koebe uniformization theorem}, which states that any closed oriented Riemannian surface is conformally equivalent to one with constant Gaussian curvature.
\\\\
The purpose of this article is to translate the Yamabe problem for two-dimensional manifolds into the setting of polyhedral surfaces. The essential ingredient of the translation is the introduction of a new discretization of Gaussian curvature.
\\\\
Defining the discrete Gaussian curvature requires a little preparation. We characterize a~\emph{polyhedral} or a \emph{piecewise flat surface} by a triple~$(S,V,d)$, where $S$ is the underlying topological surface, $d$ denotes the \emph{PL-metric} (PL stands for piecewise linear), and $V\subseteq S$ is a finite set containing the conical singularities of~$d$. 
\\\\
Let $\alpha_i$ denote the cone angle at a point~$i\in V$. The \emph{angle defect},
\begin{align*}
	W:V\to \R, \qquad W_i:=2\pi - \alpha_i,
\end{align*}
evaluates for each $i\in V$ how far the piecewise flat surface is from being flat at a neighborhood of~$i$. This notion, introduced by Tullio Regge~\cite{regge}, is best understood as the discretization of the Gaussian curvature two-form.
\\\\
The \emph{Voronoi cell} of a point $i\in V$ consists of all points on the piecewise flat surface~$(S,V,d)$ that are as close or closer to $i$ than to any other point in $V$. It arises as a natural neighborhood of the point~$i$.
\begin{definition2}\label{def:disc_Gauss_curvature} 
	The \textbf{discrete Gaussian curvature} at a point~$i\in V$ is the quotient of the angle defect~$W_i$ and the area~$A_i$ of the Voronoi cell of~$i$:
	\begin{align*}
	K:V\to\R,\qquad i\mapsto K_i:=\frac{W_i}{A_i}.
	\end{align*}
\end{definition2}

The discrete Gaussian curvature shares the following characteristic properties with the smooth Gaussian curvature: it is defined intrinsically, it satisfies the Gauss-Bonnet formula, and it scales by a factor of~$\frac{1}{r^2}$ upon a global rescaling of the metric by factor~$r$. The latter characteristic is perhaps of the biggest contribution, since the formula most commonly used for discrete Gaussian curvature --- the angle defect --- is scaling invariant.
\\\\
\emph{Discrete Yamabe problem} asks if for every PL-metric there exists a discrete conformally equivalent one with constant discrete Gaussian curvature. It can be answered affirmatively by the following theorem.
\begin{theorem2}[Discrete uniformization theorem]\label{thm:main}
	For every PL-metric~$d$ on a marked surface~$(S,V)$, there exists a discrete conformally equivalent PL-metric~$\tilde{d}$ such that the piecewise flat surface~$(S,V,\tilde{d})$ has constant discrete Gaussian curvature. 
\end{theorem2}
The proof of Theorem~\ref{thm:main} presented here is variational in nature. We translate the problem into a non-convex optimization problem with inequality constraints, which we solve using a classical theorem from calculus.
\\\\
The PL-metric~$\tilde{d}$ of constant curvature from Theorem~\ref{thm:main} is, in general, not unique.
\\\\
\emph{Discrete conformal equivalence} for piecewise flat surfaces with a fixed triangulation was introduced by Martin Ro\v{c}ek and Ruth Williams \cite{RocekWilliams}, and Feng Luo~\cite{luo}, and is a straightforward discretization of the conformal equivalence on smooth surfaces. Recall that two Riemannian metrics~$g$ and~$\tilde{g}$ on a surface~$S$ are conformally equivalent if there exists a smooth function~$u$ on~$S$ such that
$$\tilde{g} =e^{2u}g.$$
To discretize conformal equivalence, triangulate the piecewise flat surface~$(S,V,d)$ such that~$V$ is the set of vertices and every edge~$e\in E$ is a geodesic. The metric~$d$ is then uniquely determined by the edge lengths
$$\ell: E\to \R_{> 0}, \qquad \ell_{ij} = d(i,j).$$
Two PL-metrics on a surface with a fixed triangulation are \emph{discrete conformally equivalent} if their edge lengths~$\ell,\tilde{\ell}:E\to\R_{>0}$ are related by a factor~$u:V\to\R$:
$$\tilde{\ell}_{ij} = \exp \left(\frac{u_i + u_j}{2}\right)\ell_{ij}.$$
We work with a generalization of discrete conformal equivalence to piecewise flat surfaces (Definition~\ref{def:conformal_equivalence}) introduced by Alexander Bobenko, Ulrich Pinkall and Boris Springborn in~\cite[Definition 5.1.4]{bosa}. This generalization 
reveals that hyperbolic geometry is the right setting for problems involving discrete conformal equivalence. The essential relation between piecewise flat surfaces and its hyperbolic equivalent -- decorated hyperbolic surfaces with cusps -- has been explored and described in detail by Boris Springborn in~\cite{bo}.
\\\\
Another formulation of the discrete Yamabe problem for polyhedral surfaces due to Feng Luo~\cite{luo} asks for the existence of PL-metrics with a constant \emph{angle defect} within a discrete conformal class. It was solved affirmatively by Gu et al.~\cite{guluo1}, as well as by Boris Springborn~\cite{bo}. For surfaces of genus one, Luo's and our formulation of the discrete Yamabe problem are indeed equivalent. However, we believe that Luo's formulation is not a suitable discretization of the smooth Yamabe problem in general, since the angle defect is not a proper discretization of the smooth Gaussian curvature. This claim is supported by the discussions by Bobenko et al. in~\cite[Appendix B]{bosa} and by Huabin Ge and Xu Xu in~\cite[Section 1.2]{gexu}.
\\\\
This article is organized as follows. In Section~\ref{sec:preliminaries} we revise the basic concepts and provide a dictionary between piecewise flat surfaces and decorated hyperbolic surfaces with cusps. Section~\ref{sec:counterexamples} is devoted to the discussion of (non)-uniqueness of PL-metrics with constant discrete Gaussian curvature. In Section~\ref{sec:variational_principles} we translate the statement of Theorem~\ref{thm:main} into a non-convex optimization problem with inequality constraints. In Section~\ref{sec:existence} we prove Theorem~\ref{thm:main}.

\section{Fundamental definitions and results}\label{sec:preliminaries}
In this section we explain the correspondence between piecewise flat surfaces and decorated hyperbolic surfaces with cusps. Since the results in this section are well-known, we only refer to the proofs.
\\\\
Throughout the article we work with a closed oriented topological surface~$S$ and a non-empty finite set~$V\subseteq S$ of \emph{marked points}. A \emph{triangulation of the marked surface}~$(S,V)$ is a triangulation of~$S$ with the vertex set equal to~$V$. We denote a triangulation by~$\Delta$ and the \emph{set of edges} and \emph{faces} of~$\Delta$ by~$E_\Delta$ and~$F_\Delta$, respectively.
\\\\
A metric $d$ on $(S,V)$ is called \emph{piecewise linear} or a \emph{PL-metric} if is flat everywhere but on a finite set of points contained in $V$, where it develops conical singularities.
A \emph{geodesic triangulation} of the piecewise flat surface~$(S,V,d)$ is any triangulation of~$(S,V)$ where the edges are geodesics with respect to the metric~$d$.
\subsection{Tessellations of piecewise flat surfaces, discrete metric}
\paragraph{Voronoi tessellation}
Every piecewise flat surface~$(S,V,d)$ possesses a unique \emph{Voronoi tessellation}. For~$p\in S$ let~$d(p,V)$ denote the distance of~$p$ to the set~$V$, and let~$\Gamma_V(p)$ be the set of all geodesics realizing this distance. The open 2-, 1- and 0-cells of the Voronoi tessellation of~$(S,V,d)$ are the connected components of 
$$\{ p\in S\mid \abs{\Gamma_V(p)}=1 \},$$
$$\{ p\in S\mid \abs{\Gamma_V(p)}=2 \},\quad \text{and}\quad\{ p\in S\mid \abs{\Gamma_V(p)}\geq 3 \},$$
respectively. We denote the closure of the open Voronoi 2-cell containing~$i\in V$ by~$V_i$.
\paragraph{Delaunay tessellation and triangulation}
\emph{Delaunay tessellation }of a piecewise flat surface is the dual of the Voronoi tessellation. A \emph{Delaunay triangulation} arises from the Delaunay tessellation by adding edges to triangulate the non-triangular faces. \\\\
Let~$\Delta$ be a geodesic triangulation of a piecewise flat surface~$(S,V,d)$. The edge $ij\in E_\Delta$ is called a \emph{Delaunay edge} if the vertex $l$ of the adjacent triangle~$ijl\in F_\Delta$ is not contained in the interior of the circumcircle of the other adjacent triangle $ijk\in F_\Delta$. 
\begin{proposition2}
	A geodesic triangulation of a piecewise flat surface is Delaunay if and only if each of its edges is Delaunay.
\end{proposition2}
For proof see for example \cite[Proposition 10]{bosa2}.
\paragraph{Discrete metric} Let~$\Delta$ be a triangulation of the marked surface~$(S,V)$. 
\begin{definition2}\label{def:metric}
	A \textbf{discrete metric} on~$(S, V, \Delta)$ is a function $$\ell:E_\Delta\to\R_{>0}, \qquad \ell(ij) = \ell_{ij},$$
	such that for every triangle $ijk\in F_\Delta$, the (sharp) triangle inequalities are satisfied. That is,
	\begin{align*}
	\ell_{ij} + \ell_{jk} > \ell_{ki}, \quad \ell_{jk} + \ell_{ki} > \ell_{ij}, \quad \ell_{ki} + \ell_{ij} > \ell_{jk}.
	\end{align*}
	The logarithm of the discrete metric~$\ell$,
	\begin{align}\label{eq:log_lengths}
	\lambda_{ij} = 2\log\ell_{ij},
	\end{align}
	is called the\textbf{ logarithmic lengths}.
\end{definition2}
\begin{fact}\label{fact:discrete_metrics_induce_PL-metrics}
	Let~$\Delta$ be a geodesic triangulation of the piecewise flat surface~$(S,V,d)$. Then the PL-metric~$d$ induces a discrete metric on~$(S,V,\Delta)$ by measuring the lengths of the edges in~$E_\Delta$.\\\\
	Vice versa, each discrete metric~$\ell$ on a marked triangulated surface~$(S,V,\Delta)$ induces a PL-metric on~$(S,V)$.
\end{fact}
Indeed,~$\ell$ imposes a Euclidean metric on each triangle~$ijk \in F_\Delta$ by transforming it into a Euclidean triangle with edge lengths~$\ell_{ij},\ell_{jk},\ell_{ki}$. The metrics on two neighboring triangles fit isometrically along the common edge. Thus, by gluing each pair of neighboring triangles in~$\Delta$ along their common edge we equip the marked surface with a PL-metric.

\subsection{Hyperbolic metrics, ideal tessellations, and Penner coordinates} 

Consider a marked surface~$(S,V)$ equipped with a complete finite area hyperbolic metric~$d_{hyp}$ with cusps at the marked points.
We \emph{decorate} the surface~$(S,V,d_{hyp})$ with a horocycle~$\mathcal{H}_i$ at each cusp~$i\in V$. Each horocycle is small enough such that, altogether, the horocycles bound disjoint cusp neighborhoods. The set of all horocycles decorating~$(S,V,d_{hyp})$ is denoted by~$\mathcal{H}$.

\paragraph{Ideal Delaunay tessellations and triangulations}
\begin{definition2}
	An \textbf{ideal Delaunay tessellation} of a decorated hyperbolic surface $(S,V,d_{hyp},\mathcal{H})$ is an ideal geodesic cell decomposition of~$(S,V,d_{hyp})$, such that for each face~$f$ of the lift of~$(S,V,d_{hyp})$ to the hyperbolic plane~$H^2$ via an isometry of the universal cover, the following condition is satisfied. There exists a circle that touches all lifted horocycles anchored at the vertices of~$f$ externally and does not meet any other lifted horocycles.\\\\
	An \textbf{ideal Delaunay triangulation} is any refinement of an ideal Delaunay tessellation by decomposing the non-triangular faces into ideal triangles by adding geodesic edges.
\end{definition2}

\begin{theorem2}[{\cite[Theorem~4.3]{bo}}]\label{thm:ex_and_uniq_ideal_Del_decomposition}
	For each decorated hyperbolic surface with at least one cusp, there exists a unique ideal Delaunay tessellation.
\end{theorem2}

Let~$\Delta$ be a geodesic triangulation of a decorated hyperbolic surface~$(S,V,d_{hyp},\mathcal{H})$. An edge~$ij\in E_\Delta$ is called \emph{Delaunay} if the circle touching the horocycles at vertices~$i,j,k$ of one adjacent triangle~$ijk\in F_\Delta$ and the horocycle at vertex~$l$ of the other adjacent triangle~$ijl\in F_\Delta$ are externally disjoint or externally tangent. We illustrate the difference between a Delaunay and a non-Delaunay edge in Figure~\ref{fig:hyperbolic_Delaunay_edge}.
\begin{figure}[h!]
	\centering
	\begin{subfigure}[b]{0.42\textwidth}
		\labellist
		\small\hair 2pt
		\pinlabel {$i$} [ ] at 24 137
		\pinlabel {$k$} [ ] at 106 476
		\pinlabel {$j$} [ ] at 444 458
		\pinlabel {$l$} [ ] at 481 91
		\endlabellist
		\centering
		\includegraphics[width=0.7\textwidth]{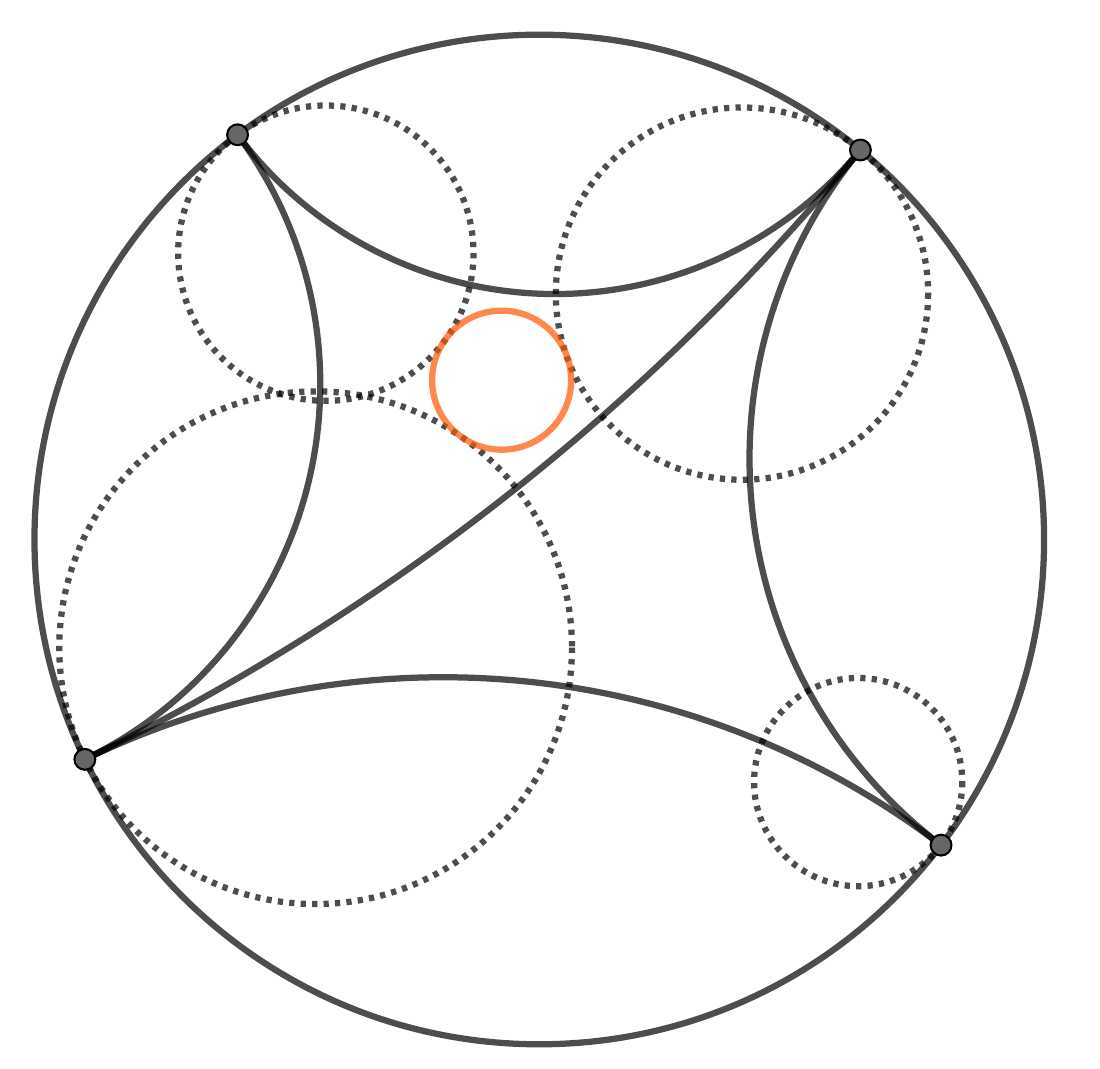}
		\caption{ The orange circle and the horocycle at~$l$ are disjoint, the edge~$ij$ is Delaunay. }
		\label{fig:Delaunay_edge_hyperbolic}
	\end{subfigure}
	\quad
	\begin{subfigure}[b]{0.42\textwidth}
		\labellist
		\small\hair 2pt
		\pinlabel {$i$} [ ] at 24 137
		\pinlabel {$k$} [ ] at 106 476
		\pinlabel {$j$} [ ] at 444 468
		\pinlabel {$l$} [ ] at 491 91
		\endlabellist
		\centering
		\includegraphics[width=0.7\textwidth]{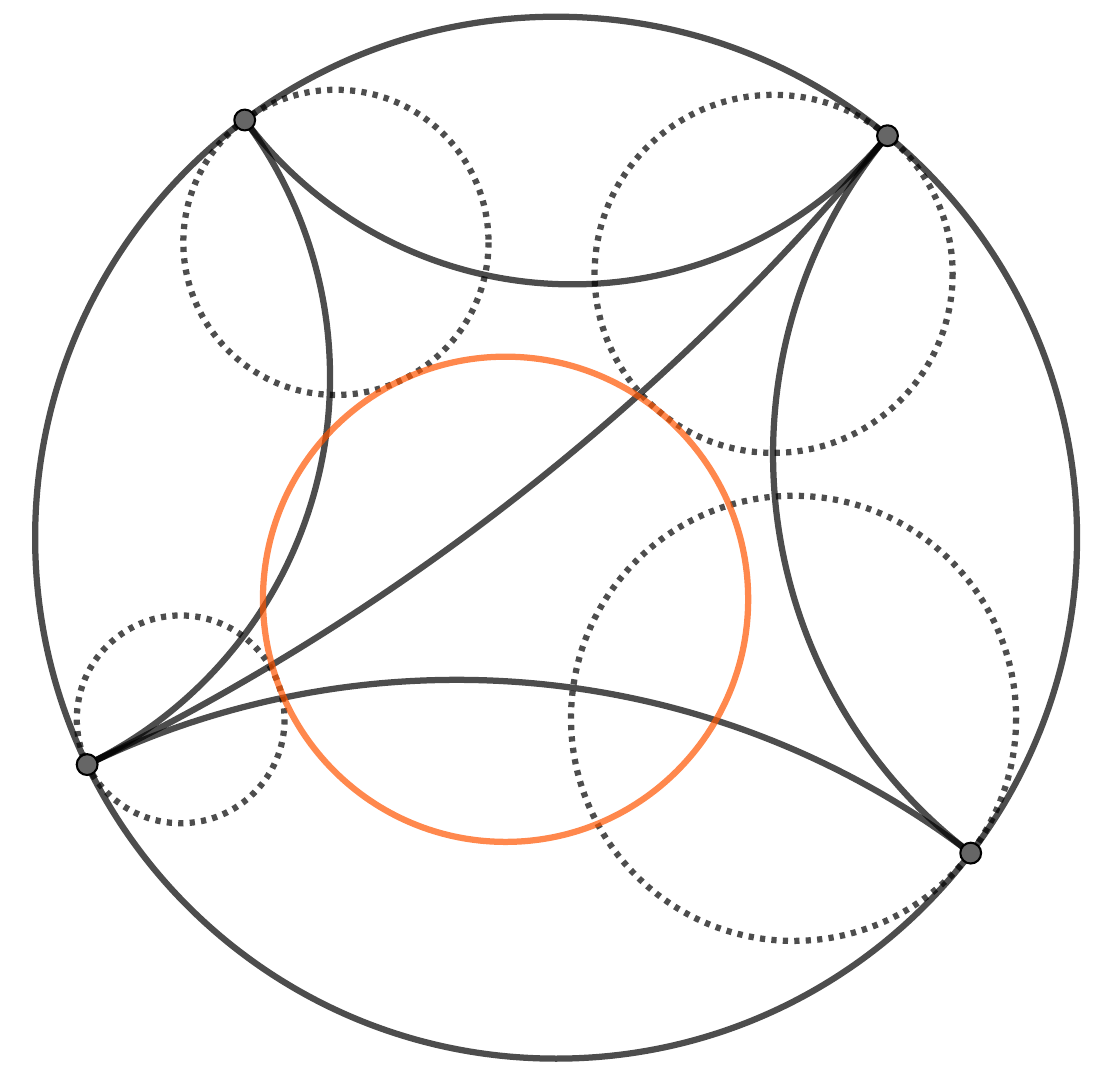}
		\caption{ The orange circle and the horocycle at~$l$ intersect, the edge~$ij$ is not Delaunay.  }
		\label{fig:Delaunay_edge_not_hyperbolic}
	\end{subfigure}
	\caption{ A Delaunay and a non-Delaunay edge. }
	\label{fig:hyperbolic_Delaunay_edge}
\end{figure}
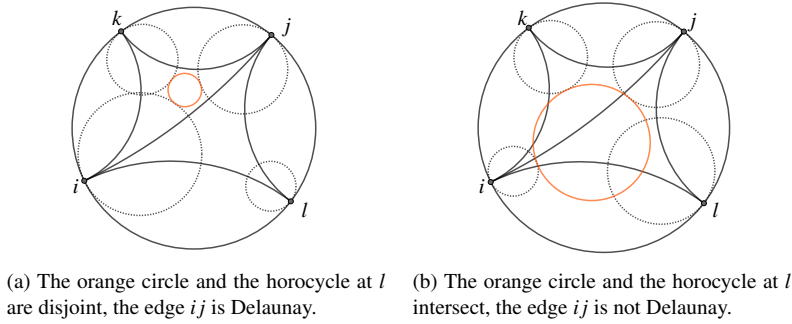

\begin{proposition2}[{\cite[Theorem~4.7]{bo}}]
	An ideal geodesic triangulation of a decorated hyperbolic surface is Delaunay if and only if each of its edges is Delaunay.
\end{proposition2}

\paragraph{Penner coordinates}Penner coordinates, introduced by Robert Penner in \cite{penner}, are the analogue of the discrete metric (see Definition~\ref{def:metric}) for decorated hyperbolic surfaces. 
\begin{definition2}\label{def:signed_horocycle_distance}
	Let~$i$ and~$j$ be two ideal points of the hyperbolic plane. Let~$\mathcal{H}_i$ and~$\mathcal{H}_j$ be two horocycles, anchored at ideal points~$i$ and~$j$, respectively. The \textbf{signed horocycle distance} between~$\mathcal{H}_i$ and~$\mathcal{H}_j$ is the length of the segment of the geodesic line connecting the cusps~$i$ and~$j$, truncated by the horocycles. The length is taken negative if~$\mathcal{H}_i$ and~$\mathcal{H}_j$ intersect.
\end{definition2}
The signed distances between horocycles of a decorated ideal hyperbolic triangle are illustrated in Figure~\ref{fig:penner_coordinates_triangle}. The distance~$\lambda_{ij}$ is negative, whereas the distances~$\lambda_{jk}$ and~$\lambda_{ki}$ are positive.
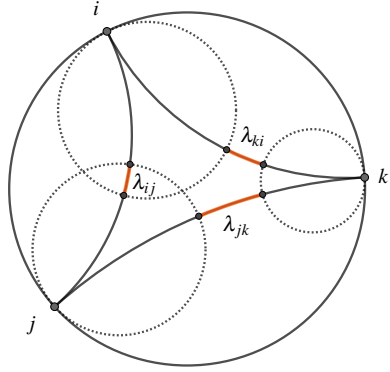
\begin{figure}[htb]
	\labellist
	\small\hair 2pt
	\pinlabel {$i$} [ ] at 146 488
	\pinlabel {$j$} [ ] at 62 74
	\pinlabel {$k$} [ ] at 518 270
	\pinlabel {$\lambda_{ki}$} [ ] at 349 321
	\pinlabel {$\lambda_{jk}$} [ ] at 328 203
	\pinlabel {$\lambda_{ij}$} [ ] at 209 258
	\endlabellist
	\centering
	\includegraphics[width=0.45\textwidth]{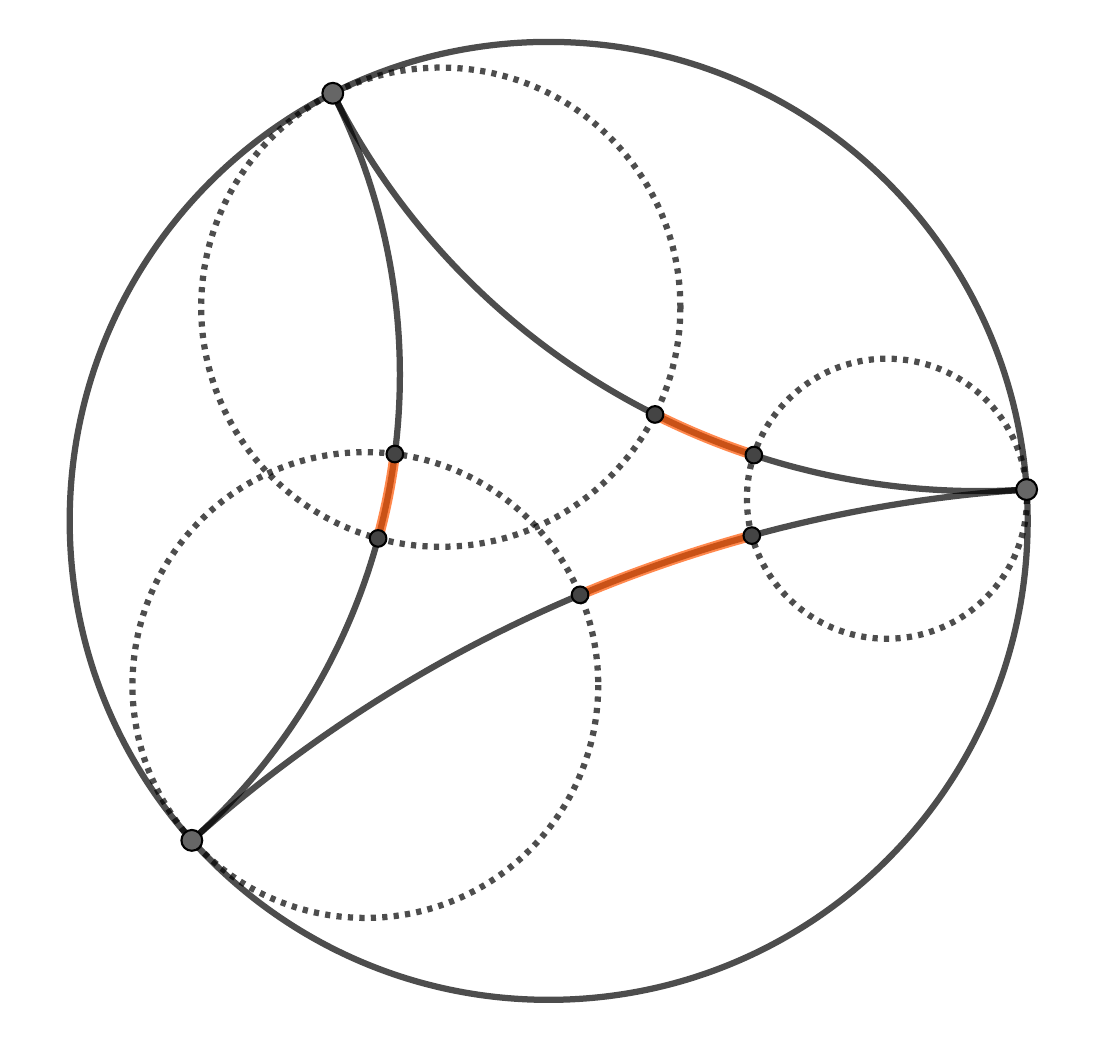}
	\caption{Penner coordinates of a decorated ideal hyperbolic triangle~$ijk$, in the Poincar\'{e} disc model.}
	\label{fig:penner_coordinates_triangle}
\end{figure}

\begin{definition2}\label{def:Penner_coordinates}
	\textbf{Penner coordinates} is a pair consisting of a triangulation~$\Delta$ of~$(S,V)$ and a map
	$$\lambda:E_\Delta\to \R, \qquad ij\mapsto \lambda_{ij}.
	$$
\end{definition2}
\begin{fact}\label{fact:hyperbolic_metric_with_cusps_from_Penner_coords}
	Penner coordinates~$(\Delta,\lambda)$ on a marked surface~$(S,V)$ define a decorated hyperbolic surface~$(S,V,d_{hyp}, \mathcal{H})$, such that the signed distance between the horocycles~$\mathcal{H}_i$ and~$\mathcal{H}_j$, with~$ij\in E_\Delta$, is~$\lambda_{ij}$.
	\\\\
	Vice versa, let~$\Delta$ be a geodesic triangulation of a decorated hyperbolic surface $(S,V,d_{hyp}, \mathcal{H})$. Then~$(S,V,d_{hyp},\mathcal{H})$ induces Penner coordinates~$(\Delta,\lambda)$ by measuring the signed horocycle distance between horocycles~$\mathcal{H}_i$ and~$\mathcal{H}_j$ for each~$ij\in E_\Delta$.
\end{fact}

\subsection{From piecewise flat surfaces to decorated hyperbolic surfaces and back again}
Piecewise flat surfaces and decorated hyperbolic surfaces are, in fact, equivalent structures. 
\begin{theorem2}[{\cite[Theorem~4.12]{bo}}]\label{thm:ideal_Del_and_Del_triangulations}
	Let~$(S,V)$ be a marked surface with a triangulation~$\Delta$.
	\\\\
	Let~$\ell:E_\Delta\to \R_{> 0}$ be a discrete metric on $(S,V,\Delta)$ such that~$\Delta$ is a Delaunay triangulation of the piecewise flat surface~$(S,V,d_{\ell})$.  Let~$\lambda$ be the logarithmic lengths of~$\ell$ defined by Equation~\eqref{eq:log_lengths}. Then~$\Delta$ is an ideal Delaunay triangulation of the decorated hyperbolic surface defined on the marked surface~$(S,V)$ by Penner coordinates~$(\Delta, \lambda)$.
	\\\\
	Vice versa, let~$(\Delta, \lambda)$ be Penner coordinates on~$(S,V)$ such that~$\Delta$ is an ideal Delaunay triangulation of the decorated hyperbolic surface defined on~$(S,V)$ by~$(\Delta, \lambda)$. Then the map~$\ell:E_\Delta\to \R_{\geq 0}$, defined by Equation~\eqref{eq:log_lengths}, is a discrete metric on~$(S,V,\Delta)$, and~$\Delta$ is a Delaunay triangulation of the polyhedral surface~$(S,V,d_\ell)$.
\end{theorem2}

\subsection{Discrete conformal classes}
Theorem~\ref{thm:ideal_Del_and_Del_triangulations} tells us that each piecewise flat surface induces a decorated hyperbolic surface, and vice versa.
\begin{definition2}\label{def:conformal_equivalence}
	Two PL-metrics on a marked surface~$(S,V)$ are \textbf{discrete conformally equivalent} if the two induced decorated hyperbolic surfaces are isometric, through a map $\varphi$, where $\varphi$ is homotopic to the identity in $S-V$ relative to $V$.
\end{definition2}
Discrete conformal equivalence is an equivalence relation on the space of PL-metrics of a marked surface~$(S,V)$. The corresponding equivalence classes are called \emph{conformal classes}. In particular, discrete conformally equivalent PL-metrics induce \textbf{different} decorations on the -- up to isometry -- \textbf{same} hyperbolic surface.
\\\\
Let~$d$ and~$\tilde{d}$ be two discrete conformally equivalent PL-metrics on~$(S,V)$, and let~$\mathcal{H}$ and~$\tilde{\mathcal{H}}$ denote the two decorations induced on the hyperbolic surface~$(S,V,d_{hyp})$ by~$d$ and~$\tilde{d}$, respectively.
\\\\
Let~$u_i$ denote the signed distance from the horocycle~$\mathcal{H}_i$ to the horocycle~$\tilde{\mathcal{H}}_i$. The distance is taken positive if~$\tilde{\mathcal{H}}_i$ is closer to the cusp at~$i$ than~$\mathcal{H}_i$  -- as illustrated in Figure~\ref{fig:distance_horocycles} in the halfplane model -- and negative otherwise.
The map
$$u:V\to\R, \qquad i\mapsto u_i,$$
is called a \emph{conformal factor}, or a \emph{conformal change} from $d$ to~$\tilde{d}$.
\\
\begin{wrapfigure}{r}{0.25\textwidth}
	\vspace{-22pt}
	\begin{center}
		\labellist
		\small\hair 2pt
		\pinlabel {$\mathcal{H}_i$} [ ] at 229 108
		\pinlabel {$\tilde{\mathcal{H}}_i$} [ ] at 229 230
		\pinlabel {$u_i$} [ ] at 72 135
		\pinlabel {$i$} [ ] at 152 250
		\endlabellist
		\centering
		\includegraphics[width=0.23\textwidth]{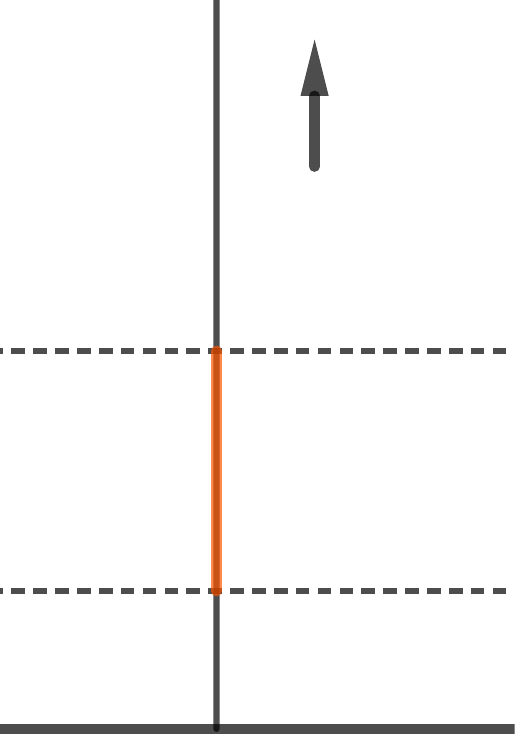}
	\end{center}
	\vspace{-15pt}
	\caption{}
	\label{fig:distance_horocycles}
\end{wrapfigure}

The position of each horocycle in~$\tilde{\mathcal{H}}$ is completely determined by the decorated hyperbolic surface~$(S,V,d_{hyp}, \mathcal{H})$ and the conformal factor~$u$. Thus, for a fixed marked surface~$(S,V)$, each PL-metric~$\tilde{d}$ in the conformal class of the PL-metric~$d$ is uniquely defined by~$d$ and the conformal factor~$u$.
\\\\
To express this relation, we denote PL-metric~$\tilde{d}$ and the decoration~$\tilde{\mathcal{H}}$ by~$d(u) $ and $\mathcal{H}(u)$, respectively. Further, if~$\tilde{\Delta}$ is a Delaunay triangulation of~$(S,V,\tilde{d})$, the Penner coordinates~$(\tilde{\Delta},\tilde{\lambda})$ are denoted by~$(\Delta(u),\lambda(u))$.
\\\\
Vice versa, each conformal factor defines a PL-metric in the conformal class of~$d$. In other words:

\begin{proposition2}\label{prop:conformal_classes_parametrisation}
	The conformal class of the piecewise flat surface~$(S,V,d)$ is parametrized by the vector space
	$$ \R^V=\{ u:V\to \R \}.$$
\end{proposition2}
As shown by Robert Penner in~\cite{penner}, the vector space $\R^V$ admits a cell decomposition into Penner cells.
\begin{definition2}\label{def:Penner_sets}
	Let~$(S,V,d)$ be a piecewise flat surface, and let~$\Delta$ be a triangulation of the marked surface~$(S,V)$. The \textbf{Penner cell} of~$\Delta$ in the conformal class of~$(S,V,d)$ is the set
	$$\mathcal{A}_\Delta = \{ u\in \R^V \mid \Delta \text{ is a Delaunay triangulation of }(S,V,d(u)) \}.$$
\end{definition2}

The set of all triangulations with non-empty Penner cells in the conformal class of~$(S,V,d)$ is denoted by~$\mathfrak{D}(S,V,d)$.
%
%
\\\\
Discrete conformal equivalence also induces a relation on discrete metrics.
\begin{proposition2}\label{prop:equivalence_disc_conformal_fixed_triangulation}
	Let~$d$ and~$\tilde{d}$ be two conformally equivalent PL-metrics on a marked surface~$(S,V)$, related by the conformal factor~$u:V\to \R$, and let~$\Delta$ be a geodesic triangulation of the surface~$(S,V,d)$, as well as the surface~$(S,V,\tilde{d})$. Then the discrete metrics~$\ell$ and~$\tilde{\ell}$, induced by~$d$ and~$\tilde{d}$, respectively, satisfy
	\begin{equation*}
	\tilde{\ell}_{ij} = \ell_{ij}e^{\frac{u_i+u_j}{2}}
	\end{equation*}
	for every edge~$ij\in E_\Delta$.
\end{proposition2}
For proof see \cite[Theorem 5.1.2]{bosa}.
\begin{remark2}
	Proposition~\ref{prop:equivalence_disc_conformal_fixed_triangulation} is the definition of discrete conformal equivalence for piecewise flat surfaces with fixed triangulation, introduced by Feng Luo~\cite{luo}.
\end{remark2}

\section{Counterexamples to uniqueness of metrics with constant curvature}\label{sec:counterexamples}
Uniqueness of PL-metrics with constant discrete Gaussian curvature up to global scaling in discrete conformal classes holds in three special cases:
\begin{itemize}
	\item \emph{$S$ is of genus zero and~$\abs{V} = 3$.}\\
	This follows from the positive semi-definiteness of the second derivative of the function~$\F$, defined in Fact~\ref{fact:alternative_vat_principle}.
	\item \emph{$S$ is of genus one.}\\
	In this case the Yamabe problem is equivalent to the discrete uniformization problem. The uniqueness follows from the positive semi-definiteness of the second derivative of function~$\E$ (Definition~\ref{def:E_function_Penner_cell}) and was proved by Xianfeng Gu, Feng Luo, Jian Sun and Tianqi Wu in \cite{guluo1}.
	\item \emph{$S$ is of genus larger than one and~$\abs{V} = 1$.}\\
	This case is trivial, since every discrete conformal class consists of one PL-metric up to a global scaling.
\end{itemize}
In order to show that uniqueness does not hold in general, we construct several examples of pairs of discrete conformally equivalent PL-metrics with constant discrete Gaussian curvature on the \emph{sphere with four marked points} -- that is, a tetrahedron -- and the \emph{surface of genus two with two marked points}.
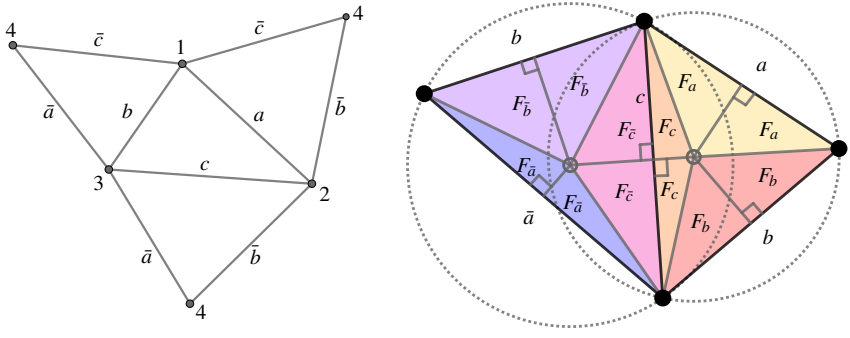
\begin{figure}[h!]
	\centering
	\begin{subfigure}[b]{0.4\textwidth}
		\labellist
		\small\hair 2pt
		\pinlabel {$1$} [ ] at 235 354
		\pinlabel {$3$} [ ] at 126 172
		\pinlabel {$2$} [ ] at 426 157
		\pinlabel {$4$} [ ] at 10 377
		\pinlabel {$4$} [ ] at 471 399
		\pinlabel {$4$} [ ] at 262 0
		\pinlabel {$b$} [ ] at 164 268
		\pinlabel {$a$} [ ] at 339 261
		\pinlabel {$c$} [ ] at 267 195
		\pinlabel {$\bar{a}$} [ ] at 60 267
		\pinlabel {$\bar{a}$} [ ] at 190 73
		\pinlabel {$\bar{b}$} [ ] at 447 274
		\pinlabel {$\bar{b}$} [ ] at 334 76
		\pinlabel {$\bar{c}$} [ ] at 126 363
		\pinlabel {$\bar{c}$} [ ] at 339 379
		\endlabellist
		\centering
		\includegraphics[width=\textwidth]{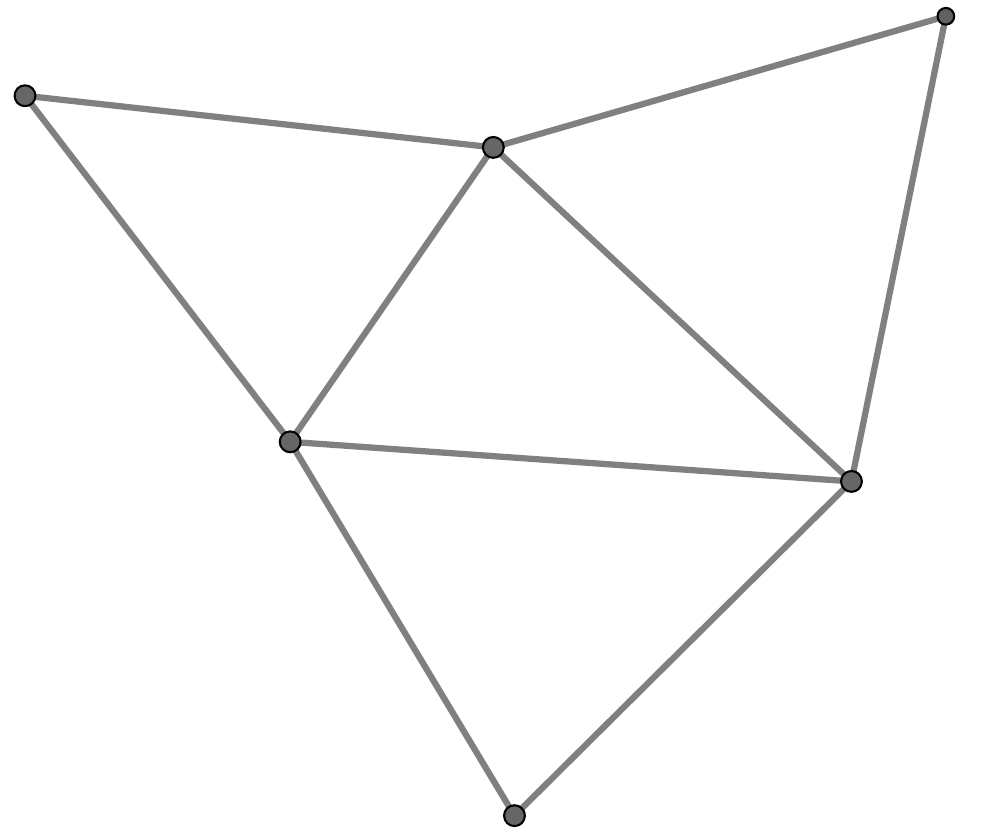}
		\\~\\~
	\end{subfigure}
	\quad
	\begin{subfigure}[b]{0.52\textwidth}
		\labellist
		\small\hair 2pt
		\pinlabel {$a$} [ ] at 391 289
		\pinlabel {$b$} [ ] at 132 320
		\pinlabel {$c$} [ ] at 263 254
		\pinlabel {$\bar{a}$} [ ] at 145 127
		\pinlabel {$b$} [ ] at 398 107
		\pinlabel {$F_a$} [ ] at 313 271
		\pinlabel {$F_a$} [ ] at 399 220
		\pinlabel {$F_b$} [ ] at 327 121
		\pinlabel {$F_b$} [ ] at 398 170
		\pinlabel {$F_c$} [ ] at 292 223
		\pinlabel {$F_c$} [ ] at 294 156
		\pinlabel {$F_{\bar{a}}$} [ ] at 192 142
		\pinlabel {$F_{\bar{a}}$} [ ] at 143 181
		\pinlabel {$F_{\bar{b}}$} [ ] at 138 246
		\pinlabel {$F_{\bar{b}}$} [ ] at 199 267
		\pinlabel {$F_{\bar{c}}$} [ ] at 249 220
		\pinlabel {$F_{\bar{c}}$} [ ] at 246 154
		\endlabellist
		\centering
		\includegraphics[width=\textwidth]{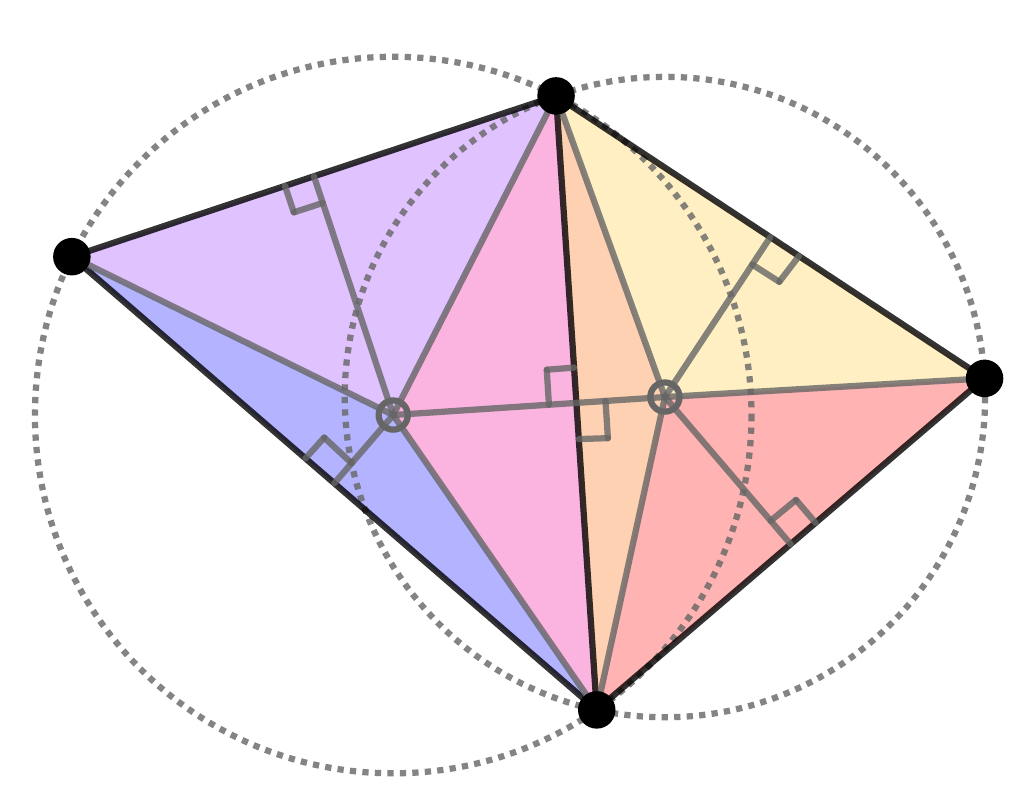}
		\label{fig:hyperbolic_example_2}
	\end{subfigure}
	\caption{A tetrahedron (left) and a division of areas in the two triangles (right). }
	\label{fig:almost_symmetric}
\end{figure}
\paragraph{Tetrahedra with constant curvature}
We start with a combinatorial tetrahedron, denoting the vertices and edges as in Figure~\ref{fig:almost_symmetric} (left). On this tetrahedron we define the PL-metric~$d_0$ by prescribing the following lengths to the edges:
$$a = \overline{a} = 1,\qquad b=\overline{b}=b_0, \qquad c = \overline{c}=c_0.$$
\begin{fact}\label{fact:Delaunay_equivalence}
	Let~$\Delta$ be a geodesic triangulation of a piecewise flat surface~$(S,V,d)$ and let $ijk$ and~$ijl $ be two neighboring triangles in~$ F_\Delta$. Let~$\alpha_k, \alpha_l$ be the angles opposite of the edge~$ij$ in the triangles~$ijk$ and~$ijl$, respectively. The edge~$ij$ is Delaunay if one of the following equivalent \textbf{Delaunay conditions} holds:
	\begin{compactenum}[a)]
		\item $\cot\alpha_k + \cot\alpha_l\geq 0$,
		\item $\alpha_k+ \alpha_l \leq\pi$,
		\item $\cos\alpha_k + \cos\alpha_l \geq 0.$
	\end{compactenum}
\end{fact}
The values of $b_0$ and $c_0$ need to be greater than 1 and chosen so that the edges of the tetrahedron are Delaunay. This is the case if and only if the triangle with edge lengths~$1,b_0, c_0$ is acute. Due to condition \emph{c)} in Fact~\ref{fact:Delaunay_equivalence}, this is further equivalent to the following inequality:
\begin{align}\label{eq:Delaunay_cond_c}
c_0^2\leq a_0^2 +b_0^2= 1+b_0^2.
\end{align}
Denoting the area of the triangle with edge lengths~$1,b_0, c_0$ by~$A$, one sees that the PL-metric~$d_0$ has constant discrete Gaussian curvature,
$$K_i = \frac{\pi}{A} \qquad\text{for }i\in\{1,\dots,4\}.$$
We now apply the following family of discrete conformal changes to~$d_0$:
\begin{align*}
u\circ v:\R\to \R^4, \qquad u(v) = (u_1,u_2,u_3,u_4)(v) :=(0,0,v,v). 
\end{align*}

\begin{lemma2}\label{lemma:equal_curvature_almost_symmetric}
	Let
	$$\mathcal{S}_{(b_0, c_0)}:=\left[ -\log(b_0^2 + c_0^2), \log(b_0^2 + c_0^2) \right].$$
	The PL-metric~$d(v)$, defined by applying the discrete conformal change~$u(v)$ to the metric~$d_0$, has Delaunay edges if~$v\in \mathcal{S}_{(b_0, c_0)}$. Its discrete Gaussian curvature at two pairs of vertices is equal,
	$$K_1 = K_2, \qquad\text{and}\qquad K_3 = K_4.$$
\end{lemma2}

\begin{proof}
	For each~$v\in\R$ the tetrahedron with metric~$d(v)$ has edge lengths
	$$a=1, \qquad b=\overline{b}=b_0e^{v/2},\qquad c=\overline{c}=c_0e^{v/2},\qquad \bar{a}=e^{v}.$$
	The tetrahedron thus consists of two triangles with edge lengths~$a,b,c$ and two triangles with edge lengths~$\bar{a}, b, c$. The equality of the curvatures follows immediately from the fact that~$W_1 = W_2, W_3 = W_4, A_1 = A_2$ and~$A_3 = A_4.$
	\\\\
	The minimal and maximal value of the parameter~$v$ follow from the properties of Delaunay edges (Fact~\ref{fact:Delaunay_equivalence}) and Equation~\eqref{eq:Delaunay_cond_c}.
\end{proof}
Lemma~\ref{lemma:equal_curvature_almost_symmetric} implies that the PL-metric~$d(v)$ has constant discrete Gaussian curvature if~$K_1 = K_3$. In order to test if, for a fixed value of~$b_0$ and~$c_0$, this equality holds, we transform it into an expression more favorable for calculations.
\\\\
Let~$A$ and~$\bar{A}$ denote the area of the triangles with side lengths~$a,b,c$ and~$\bar{a}, b,c,$ respectively, and let~$F_a,\dots, F_{\bar{c}}$ denote the areas as in Figure~\ref{fig:almost_symmetric} (right).
\begin{lemma2}\label{lemma:def_of_g}
	The PL-metric~$d(v)$ has constant discrete Gaussian curvature if and only if~$v$ is a zero of the map
	\begin{align*}
	g_{(b_0, c_0)}:\mathcal{S}_{(b_0, c_0)}\to \R, \qquad v\mapsto 2\pi(F_{\bar{a}}-F_a)+ (\alpha - \bar{\alpha})(A+\bar{A}).
	\end{align*}
\end{lemma2}
\begin{proof}
	Follows by a straightforward calculation:
	\begin{align*}
	K_1 = K_3 \iff W_1 A_3 = W_3 A_1\iff 2\pi(F_{\bar{a}}-F_a)= (\bar{\alpha} - \alpha)(A+\bar{A}).
	\end{align*}	
\end{proof}
We plotted the graphs of the function $g_{(b_0, c_0)}$ for various values of $b_0$ and $ c_0$ in Figure~\ref{fig:counterex_summary}.
\begin{figure}[htb]
	\labellist
	\small\hair 2pt
	\pinlabel {$2$} [ ] at 205 327
	\pinlabel {$2$} [ ] at 340 237
	\pinlabel {$2$} [ ] at 505 205
	\pinlabel {$2$} [ ] at 679 240
	\pinlabel {$2$} [ ] at 821 328
	\pinlabel {$1$} [ ] at 893 120
	\pinlabel {$1$} [ ] at 672 8
	\pinlabel {$1$} [ ] at 405 6
	\pinlabel {$1$} [ ] at 190 76
	\pinlabel {$1$} [ ] at 16 229
	\pinlabel {$a^4$} [ ] at 274 286
	\pinlabel {$a^3$} [ ] at 424 220
	\pinlabel {$a^4$} [ ] at 593 219
	\pinlabel {$a^3$} [ ] at 748 283
	\pinlabel {$a^2$} [ ] at 787 71
	\pinlabel {$a^1$} [ ] at 549 7
	\pinlabel {$a^2$} [ ] at 304 36
	\pinlabel {$a^1$} [ ] at 94 155
	\pinlabel {$b^1$} [ ] at 93 302
	\pinlabel {$c^1$} [ ] at 173 214
	\pinlabel {$b^2$} [ ] at 260 177
	\pinlabel {$c^2$} [ ] at 356 118
	\pinlabel {$b^3$} [ ] at 457 116
	\pinlabel {$c^3$} [ ] at 571 96
	\pinlabel {$b^4$} [ ] at 660 124
	\pinlabel {$c^4$} [ ] at 778 156
	\pinlabel {$b^1$} [ ] at 836 231
	\endlabellist
	\centering
	\includegraphics[width=\textwidth]{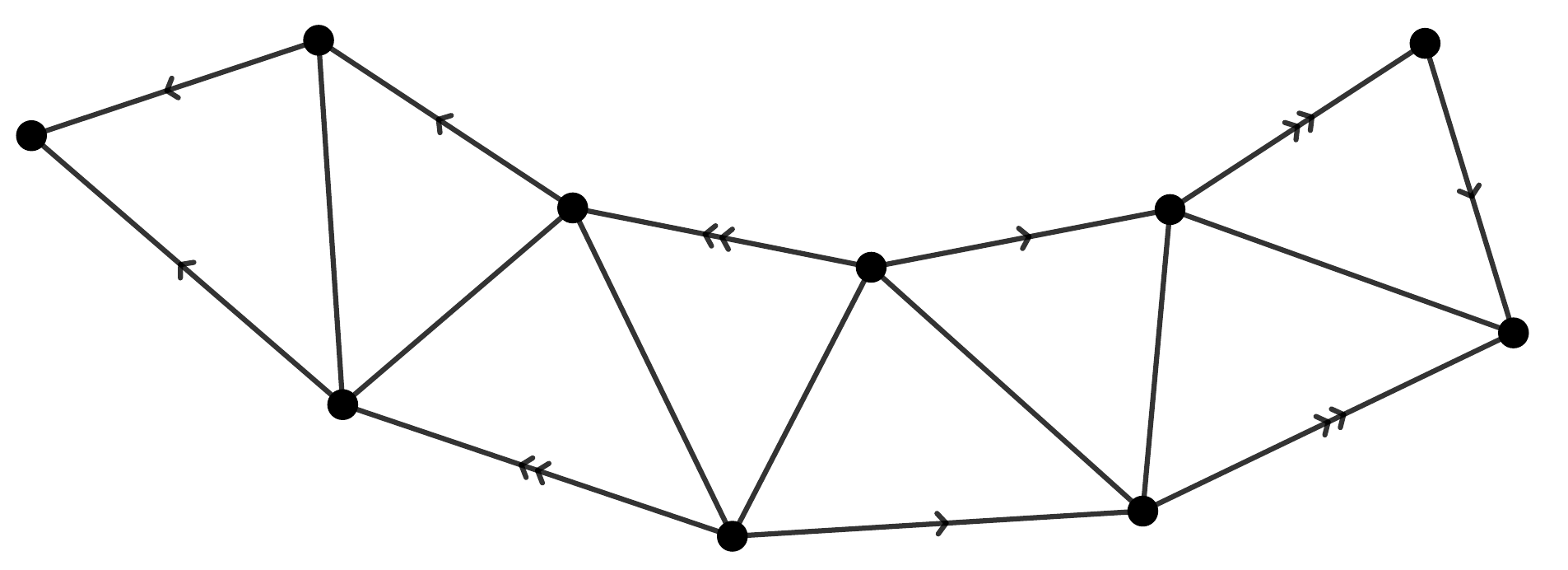}
	\caption{ }
	\label{fig:double_torus_comb}
\end{figure}
\paragraph{Surfaces of genus two with two marked points and constant curvature}
The initial metric~$d_0$ is defined on a triangulation with combinatorics as in Figure~\ref{fig:double_torus_comb}, with edge lengths prescribed as follows:
$$a^1=...=a^4 = 1,\quad b^1 = ... = b^4 = b_0,\quad c^1 = ... = c^4 = c_0,$$
for two values~$b_0,c_0 \geq 1$ satisfying Condition~\eqref{eq:Delaunay_cond_c}. As in the previous paragraph, one can easily check that~$d_0$ has constant discrete Gaussian curvature
$$K_i =- \frac{\pi}{2A},$$
where~$A$ is the area of the triangle with edge lengths~$1,b_0, c_0$. We now apply the following family of discrete conformal changes to~$d_0$:
\begin{align*}
u\circ v:\R\to \R^2, \qquad u(v) = (u_1,u_2)(v) :=(0,v). 
\end{align*}
The following lemma is the analogon of Lemmata~\ref{lemma:equal_curvature_almost_symmetric} and~\ref{lemma:def_of_g}.
\begin{lemma2}
	The PL-metric~$d(v)$, given by applying the discrete conformal change~$u(v)$ to the metric $d_0$, has Delaunay edges if~$v\in \mathcal{S}_{(b_0, c_0)}$. It has constant discrete Gaussian curvature if and only if~$v$ is a zero of the map
	\begin{align*}
	h_{(b_0, c_0)}:\mathcal{S}_{(b_0, c_0)}\to \R, \qquad v\mapsto \pi(F_{\bar{a}}-F_a)+ ( \bar{\alpha}-\alpha)(A+\bar{A}).
	\end{align*}
\end{lemma2}
\begin{proof}
	Analogous to the proofs of Lemmata~\ref{lemma:equal_curvature_almost_symmetric} and~\ref{lemma:def_of_g}.
\end{proof}
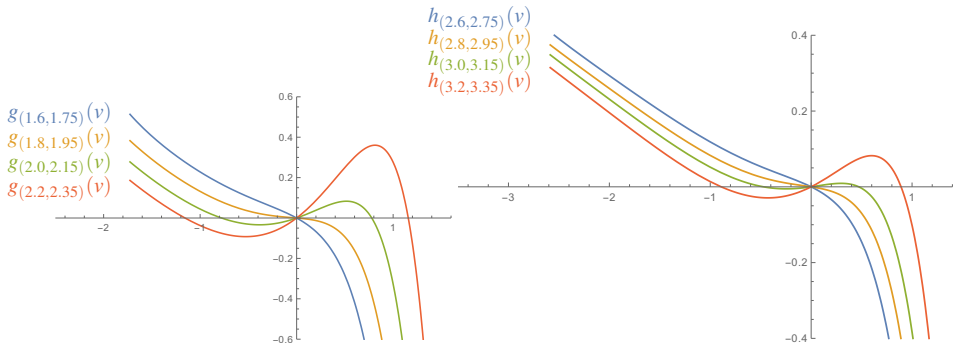
\begin{figure}[h!]
	\centering
	\begin{subfigure}[b]{0.44\textwidth}
		\labellist
		\small\hair 2pt
		\pinlabel {{\color[rgb]{0.368417, 0.506779, 0.709798} $g_{(1.6,1.75)}(v)$}} [ ] at 5 235
		\pinlabel {{\color[rgb]{0.880722, 0.611041, 0.142051} $g_{(1.8,1.95)}(v)$}} [ ] at 5 210
		\pinlabel {{\color[rgb]{0.560181, 0.691569, 0.194885} $g_{(2.0,2.15)}(v)$}} [ ] at 5 185
		\pinlabel {{\color[rgb]{0.922526, 0.385626, 0.209179} $g_{(2.2,2.35)}(v)$}} [ ] at 5 160
		\endlabellist
		\centering
		\includegraphics[width=\textwidth]{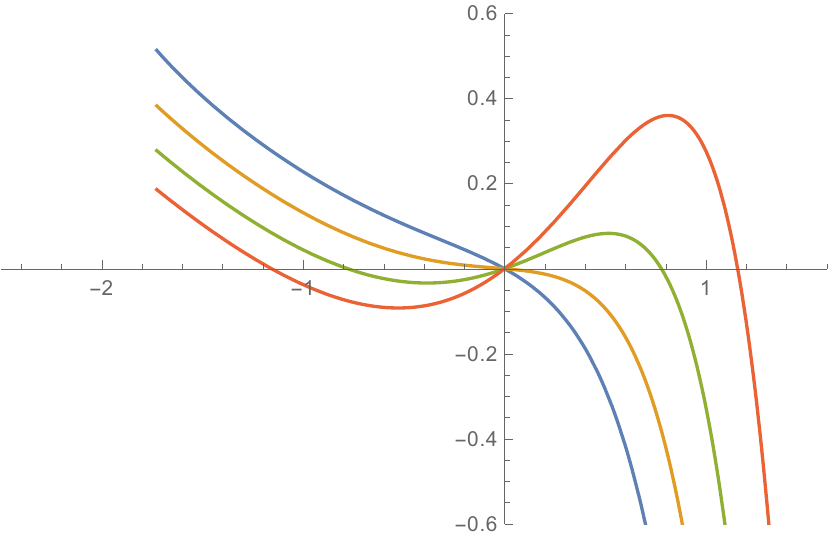}
		\label{fig:counterex_tetra_summary}
	\end{subfigure}
	\begin{subfigure}[b]{0.55\textwidth}
		\labellist
		\small\hair 2pt
		\pinlabel {{\color[rgb]{0.368417, 0.506779, 0.709798} $h_{(2.6,2.75)}(v)$}} [ ] at 21 305
		\pinlabel {{\color[rgb]{0.880722, 0.611041, 0.142051} $h_{(2.8,2.95)}(v)$}} [ ] at 21 285
		\pinlabel {{\color[rgb]{0.560181, 0.691569, 0.194885} $h_{(3.0,3.15)}(v)$}} [ ] at 21 265
		\pinlabel {{\color[rgb]{0.922526, 0.385626, 0.209179} $h_{(3.2,3.35)}(v)$}} [ ] at 21 245
		\endlabellist
		\centering
		\includegraphics[width=\textwidth]{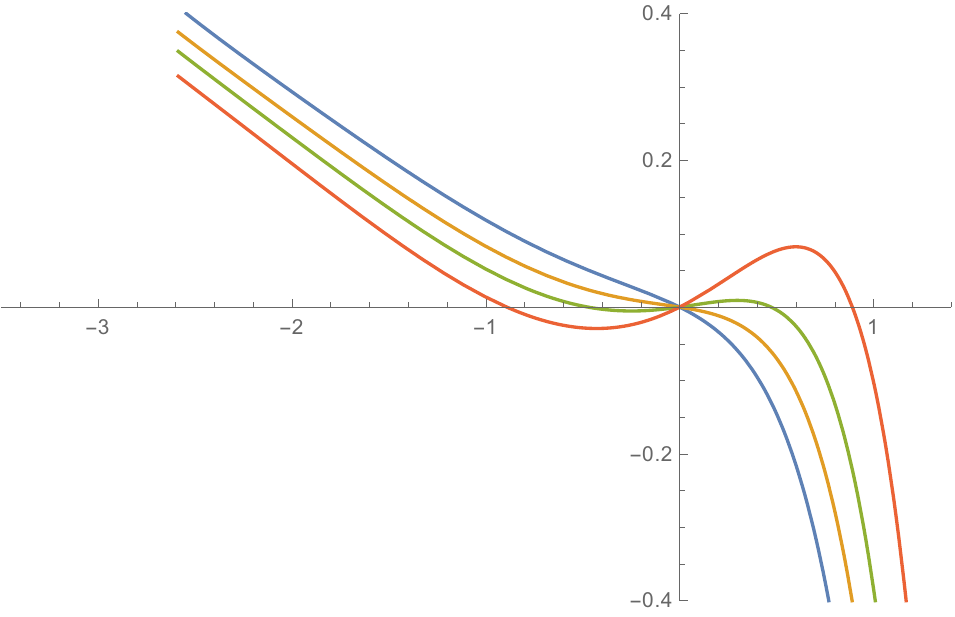}
		\label{fig:counterex_hyp_summary}
	\end{subfigure}
	\caption{Graphs of the functions~$g$ (left) and~$h$ (right) for various values of~$b_0$ and~$c_0$. }
	\label{fig:counterex_summary}
\end{figure}
The number of critical points of the maps~$g_{(b_0, c_0)}$ and~$h_{(b_0, c_0)}$ varies depending on the choice of $(b_0, c_0)$. Figure~\ref{fig:counterex_summary} illustrates the graphs of~$g_{(b_0, c_0)}$ and~$h_{(b_0, c_0)}$ for various values of~$(b_0, c_0)$. In each graph the red and green curves correspond to discrete conformal classes with more than one metric with constant discrete Gaussian curvature.

\section{Variational principles}\label{sec:variational_principles}
The goal of this article is to prove the existence of piecewise flat surfaces with constant Gaussian curvature, where the discrete Gaussian curvature is the quotient of the angle defect and the area of the corresponding Voronoi cell. In this section we translate this setting into an optimization problem which we describe by three variational principles. To this end, we define two functions -- $\mathbb{E}$ and $A_{tot}$ -- whose partial derivatives correspond to the angle defect and the area of the Voronoi cell, respectively. The functions $\mathbb{E}$ and $A_{tot}$ form the two essential building blocks of the variational principles.
\subsection{Two essential building blocks}\label{sec:two_essential_blocks}
\paragraph{The function $\mathbb{E}$}
The function~$\E$, which we will introduce shortly, was defined by Alexander Bobenko et al. in \cite{bosa}. As we will see, it is locally convex and its partial derivatives correspond to the angle defects at the vertices. Its building block is a peculiar function~$f$.
\begin{definition2}\label{def:the_function_f}
	Consider a Euclidean triangle with edge lengths~$a,b,c$ and angles $\alpha,\beta,\gamma$, opposite to edges~$a,b,c$, respectively.
	Let
	$$x=\log a,\qquad y=\log b,\qquad z=\log c,$$
	as illustrated in Figure~\ref{fig:f_function_triangle_def}. Let~$\mathfrak{A}$ be the set of all triples~$(x,y,z)\in\R^3$, such that~$(a,b,c)$ satisfy the triangle inequalities:
	$$\mathfrak{A} = \{(x,y,z)\in \R^3 \mid a+b-c>0, a-b+c>0,-a+b+c>0 \}.$$
	The function~$f$ is defined as follows:
	$$f:\mathfrak{A}\to \R, \quad f(x,y,z) = \alpha x + \beta y + \gamma z + \LobL(\alpha) + \LobL(\beta)+\LobL(\gamma),$$
	where
	$$\LobL(\alpha) = -\int_{0}^{\alpha}\log\abs{2\sin(t)}\: dt$$
	is Milnor's Lobachevsky function, introduced by John Milnor in~\cite{milnor}.
\end{definition2}
\begin{figure}[h!]
	\centering
	\begin{subfigure}[b]{0.40\textwidth}
		\labellist
		\small\hair 2pt
		\pinlabel {$\alpha$} [ ] at 133 120
		\pinlabel {$\beta$} [ ] at 433 79
		\pinlabel {$\gamma$} [ ] at 291 295
		\pinlabel {\rotatebox{352}{$c = e^z$}} [ ] at 261 40
		\pinlabel {\rotatebox{50}{$b = e^y$}} [ ] at 180 251
		\pinlabel {\rotatebox[]{305}{$a = e^x$}} [ ] at 420 220
		\endlabellist
		\centering
		\includegraphics[width=\textwidth]{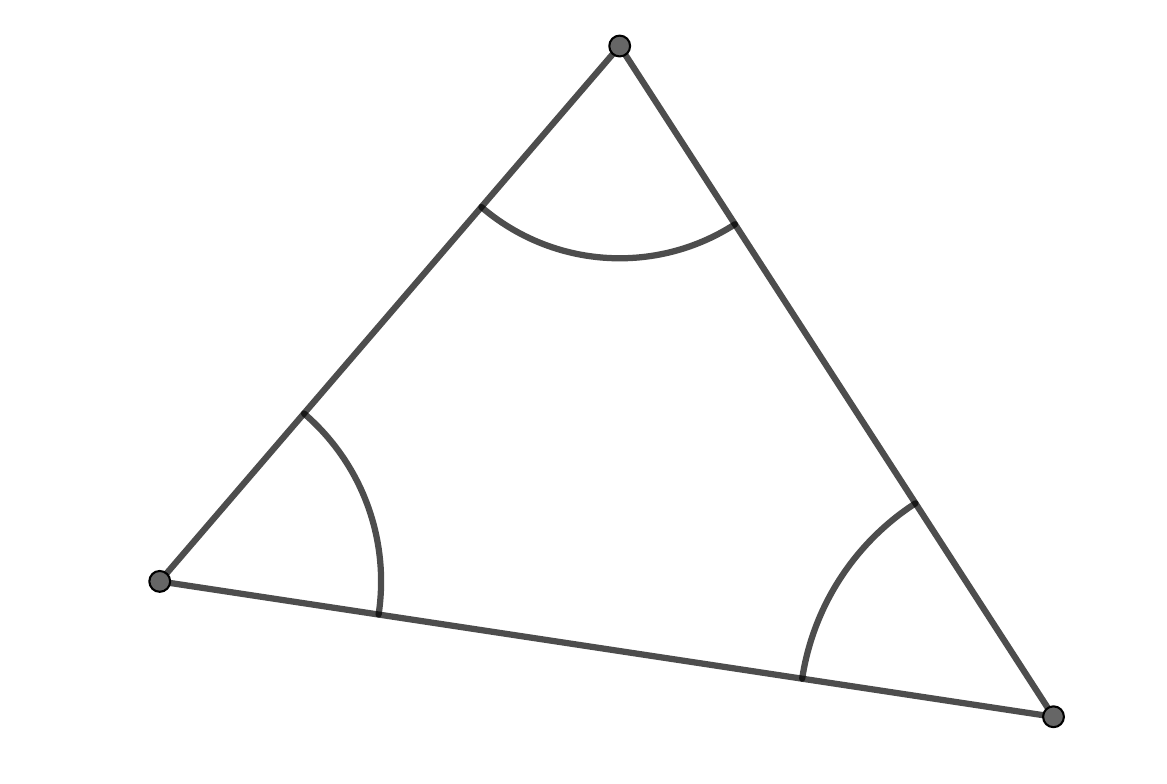}
		\caption{Logarithmic edge lengths of a triangle. }
		\label{fig:f_function_triangle_def}
	\end{subfigure}
	\quad
	\begin{subfigure}[b]{0.50\textwidth}
		\centering
		\includegraphics[width=\textwidth]{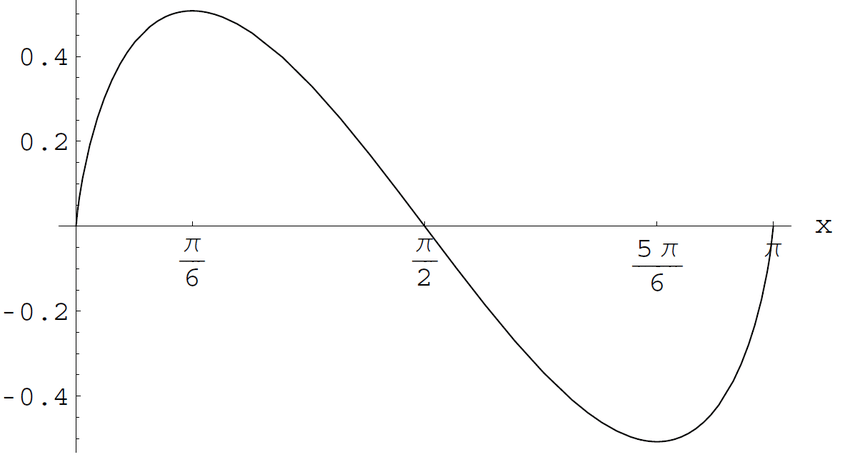}
		\caption{Graph of Milnor's Lobachevsky function,~$y=\LobL(x)$. }
		\label{fig:lobachevsky_function}
	\end{subfigure}
	\caption{ }
\end{figure}

\begin{fact}\label{fact:lobachevsky_function}
	Milnor's Lobachevsky function~$\LobL(x)$ is odd, $\pi$-periodic, and smooth except at~$x\in\pi\Z$.
\end{fact}
Recall that the discrete conformal class of a piecewise flat surface~$(S,V,d)$ is parameterized by the vector space~$\R^V$ (see Proposition~\ref{prop:conformal_classes_parametrisation}), which can be decomposed into Penner cells (see Definition~\ref{def:Penner_sets}). We first define the function~$\E_\Delta$ on each Penner cell~$\mathcal{A}_\Delta$ and then extend its domain to obtain the function~$\E$ on~$\R^V$.
\begin{definition2}\label{def:E_function_Penner_cell}
	Let~$(S,V,d)$ be a piecewise flat surface, and let~$\Delta\in\mathfrak{D}(S,V,d)$. On the Penner cell~$\mathcal{A}_{\Delta}$, the function~$\E_\Delta$ is defined as follows: 
	\begin{align*}
	&\E_{\Delta}:\mathcal{A}_\Delta\to\R,\\
	&\mathbb{E}_{\Delta}(u) = \sum_{ijk\in F_\Delta} \left( 2f\left( \frac{\tilde{\lambda}_{ij}}{2},\frac{\tilde{\lambda}_{jk}}{2},\frac{\tilde{\lambda}_{ki}}{2} \right) - \frac{\pi}{2}(\tilde{\lambda}_{ij}+\tilde{\lambda}_{jk}+\tilde{\lambda}_{ki}) \right) +2\pi \sum_{i\in V} u_i,
	\end{align*}
	where~$\tilde{\lambda}_{ij}$ are the logarithmic lengths of the discrete metric induced by the PL-metric~$d(u)$ on~$\Delta$.
\end{definition2}
\begin{lemma2}\label{lemma:E_function_partial_derivative}
	The partial derivatives of the function~$\E_\Delta$ satisfy the equation
	\begin{equation}\label{eq:E_functional_first_par_derivative}
	\pr{\E_{\Delta}}{u_i}=W_i,
	\end{equation}
	where~$W_i$ is the angle defect at vertex~$i$ of the piecewise flat surface~$(S,V,d(u))$.
\end{lemma2}
\begin{proof} Follows from~\cite[Proposition~4.1.2]{bosa}.
\end{proof}

The functions~$f$ and~$\E_{\Delta}$ have the following properties: \begin{proposition2}[Properties of~$f$ and~$\E_{\Delta}$]\label{prop:properties_E_f}
	The functions~$f$ and~$\E_{\Delta}$ are analytic and locally convex on~$\mathfrak{A}$ and~$\mathcal{A}_\Delta$, respectively. Their second derivatives are positive semidefinite quadratic forms with one-dimensional kernels, spanned by~$(1,1,1)\in \R^3$, $(1,\dots,1)\in \R^V$, respectively. Further,
	\begin{align*}
	f(x+t,y+t,z+t) &= f(x,y,z) + \pi t & \text{for all }(x,y,z)\in \mathfrak{A},\\
	\E_\Delta(u+c(1,\dots,1)) &= \E_\Delta(u) + 2\pi\chi(S)c & \text{for all }u\in \mathcal{A}_\Delta,
	\end{align*}
	where~$\chi(S)$ denotes the Euler characteristic of the surface~$S$.
\end{proposition2}
\begin{proof}
	See~\cite[Equation~(4-5)]{bosa} or~\cite[Propositions~7.2 and~7.7]{bo}.
\end{proof}

\begin{theorem2}[Extension]\label{thm:E_function_extension}
	For a conformal factor~$u\in\R^V$, let~$\Delta(u)$ be a Delaunay triangulation of the surface~$(S,V,d(u))$. The map
	$$\E:\R^V\to \R, \qquad u\mapsto \E_{\Delta(u)}(u),$$
	is well-defined and twice continuously differentiable. Its second derivative is a positive semidefinite quadratic form with one-dimensional kernel, spanned by~$(1,\dots,1)\in \R^V$. Explicitly,
	\begin{align*}
	d^2\mathbb{E}=\frac{1}{4}\sum_{ij\in E}(\cot\alpha_k^{ij} + \cot\alpha_l^{ij})(du_i-du_j)^2.
	\end{align*}
\end{theorem2}
\begin{proof}
	Follows from~\cite[Proposition~4.1.6]{bosa} and~\cite[Section 7 and 8]{bo}. 
\end{proof} 
\paragraph{The function $\mathbf{A_{tot}}$}
The function~$A_{tot}$, whose first partial derivatives correspond to the area of the Voronoi cells, denotes the total area of the surface. We first define the function~$A_{tot}^\Delta$ on each Penner cell~$\mathcal{A}_\Delta$ and then extend its domain to obtain the function~$A_{tot}$ on~$\R^V$.
\begin{definition2}\label{def:area_function_Penner_cell}
	Let~$(S,V,d)$ be a piecewise flat surface, and let~$\Delta\in\mathfrak{D}(S,V,d)$. On the Penner cell~$\mathcal{A}_{\Delta}$, the function~$A_{tot}^\Delta$ is defined as follows:
	$$A_{tot}^\Delta:\mathcal{A}_\Delta\to\R,\qquad A_{tot}^\Delta(u) = \sum_{ijk\in F_\Delta} A_{ijk}(u),$$
	where~$A_{ijk}(u)$ is the area of the triangle with vertices~$i,j,k\in V$ on the piecewise flat surface~$(S,V,d(u))$. 
\end{definition2}
Let us denote the area of the Voronoi cell of a marked point~$i\in V$ by~$A_i.$
\begin{lemma2}\label{lemma:area_function_first_derivative}
	The function~$A_{tot}^\Delta$ is analytic. Its partial derivatives satisfy the equation
	\begin{align}\label{eq:area_functional_first_par_derivative}
	\frac{\partial A_{tot}^\Delta}{\partial u_i} = 2A_i.
	\end{align}
	Its second derivative is
	\begin{align*}
	d^2A_{tot}^\Delta &= \sum_{ij\in E_\Delta}2A_{ij}(du_i+du_j)^2 - \frac{1}{2}\sum_{ij\in E_\Delta}(R_{ijk}^2\cot\alpha_k^{ij}+R_{ijl}^2\cot\alpha_l^{ij})(du_i-du_j)^2,
	\end{align*}
	where the vertices~$k,l\in V$ are the opposite vertices in the neighboring triangles~$ijk,ijl\in F_{\Delta(u)}$, $A_{ij} = \frac{\ell_{ij}^2}{8}\left( \cot \alpha_k^{ij}+ \cot \alpha_l^{ij} \right)$, and~$R_{ijk}$ denotes the radius of the circumcircle of the triangle~$ijk$.
\end{lemma2}
\begin{figure}[htb]
	\labellist
	\small\hair 2pt
	\pinlabel {$i$} [ ] at 34 181
	\pinlabel {$j$} [ ] at 417 124
	\pinlabel {$k$} [ ] at 229 414
	\pinlabel {$\alpha_i^{jk}$} [ ] at 110 209
	\pinlabel {$\alpha_j^{ki}$} [ ] at 330 178
	\pinlabel {$\alpha_k^{ij}$} [ ] at 223 337
	\pinlabel {$A_{jk}^i$} [ ] at 257 266
	\pinlabel {${\ell}_{jk}$} [ ] at 349 274
	\pinlabel {${\ell}_{ij}$} [ ] at 253 118
	\pinlabel {${\ell}_{ki}$} [ ] at 105 300
	\pinlabel {$R_{ijk}$} [ ] at 228 70
	
	\endlabellist
	\centering
	\includegraphics[width=0.3\textwidth]{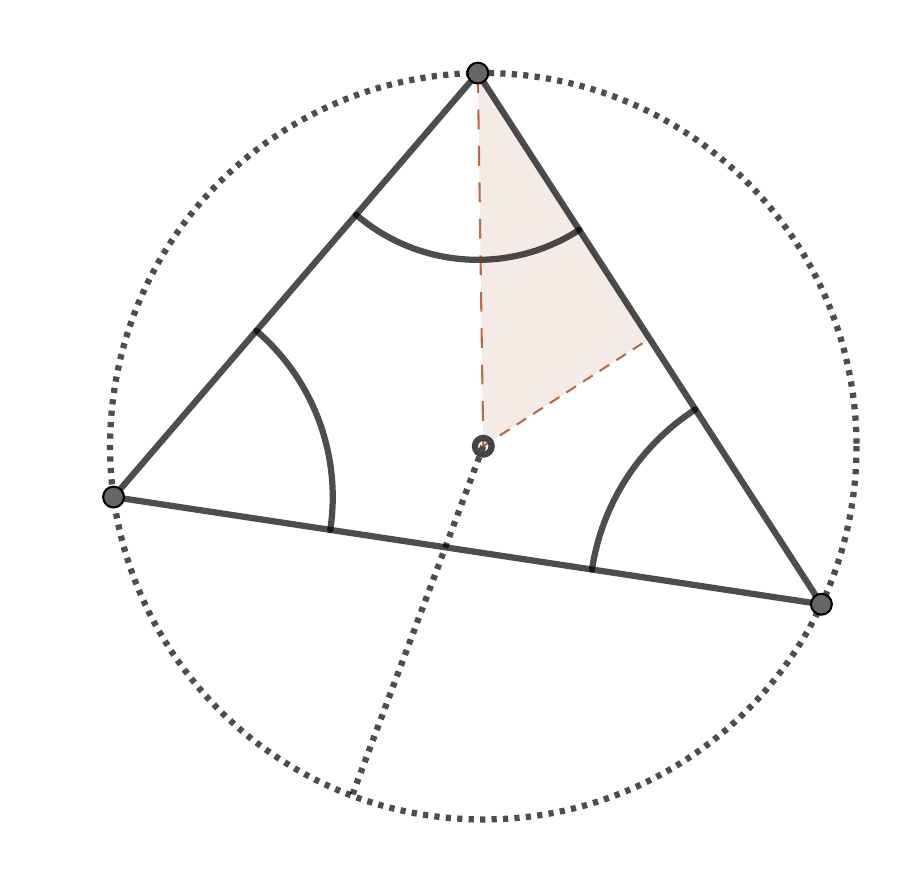}
	\caption{  }
	\label{fig:triangle_derivatives}
\end{figure}
\begin{proof}
	The function~$A_{tot}^\Delta$ is analytic since the area~$A_{ijk}(u)$ of each triangle~$ijk\in F_\Delta$ is an analytic function with respect to the vector of conformal factors $u$\footnote{This follows for example from Heron's formula, and the fact that for all triangles in $F_\Delta, A_{ijk}(u)>0$.}.
	\\\\
	Consider a triangle with vertices~$i,j,$ and~$k$, and let~$A_{jk}^i$ denote the signed area of the triangle with vertices~$k$, the circumcentre of the triangle~$ijk$, and the midpoint of the edge~$jk$, as depicted in Figure~\ref{fig:triangle_derivatives}. The sign of~$A_{jk}^i$ is positive if the circumcentre of~$ijk$ lies inside the triangle, and negative otherwise. Then 
	$$A_{jk}^i = \frac{\ell_{jk}^2}{8}\cot\alpha_i^{jk},$$
	and the area of the Voronoi cell~$V_i$ of a piecewise flat surface~$(S,V,d)$ satisfies the equation
	$$A_{i} = \sum_{jk\mid ijk\in F_\Delta}A^j_{ki} + A^k_{ij} .$$
	Thus,	
	\begin{align*}
	\frac{\partial A_{ijk}}{\partial u_i} =2A^j_{ki} + 2A^k_{ij} - R_{ijk}^2 \frac{\partial}{\partial u_i}(\underbrace{\alpha_i^{jk}+\alpha_j^{ki}+\alpha_k^{ij}}_{=\pi})=2A^j_{ki} + 2A^k_{ij}.
	\end{align*} 
	Due to the linearity of the area function,
	\begin{align*}
	\frac{\partial A_{tot}^\Delta}{\partial u_i} = \sum_{jk\mid ijk\in F_\Delta}2A^j_{ki} + 2A^k_{ij} =2A_i.
	\end{align*}
	In the upcoming calculations we use the following formula from~\cite[Equation~(4-8)]{bosa}.
	\begin{lemma*}\label{lemma:angle_derivatives}
		Let~$a,b,c$ be edge lengths of a triangle, $\alpha,\beta,\gamma$ angles opposite of $a,b,c$, respectively, and let~$\lambda_a,\lambda_b,\lambda_c$ be the logarithmic lengths.
		Then
		\begin{equation*}
		2 d\alpha = (\cot\beta+\cot\gamma)d\lambda_a - \cot\gamma d\lambda_b - \cot\beta d\lambda_c.
		\end{equation*}
	\end{lemma*}
	Since 
	\begin{align*}
	\frac{\partial A^j_{ki}}{\partial u_i} &= A^j_{ki}-\frac{1}{2}R_{ijk}^2 \frac{\partial \alpha_j^{ki}}{\partial u_i} =A^j_{ki}-\frac{1}{4}R_{ijk}^2\cot\alpha_k^{ij},
	\end{align*}
	we obtain the equation
	\begin{align*}
	\pr{^2A_{tot}^\Delta}{u_i^2} = 2A_i - \frac{1}{2}\sum_{jk\mid ijk\in F_\Delta}R_{ijk}^2(\cot\alpha_k^{ij} + \cot\alpha_j^{ki}).
	\end{align*}
	
	Let~$i,j\in V$ be two vertices. If~$j$ is not adjacent to~$i$, 
	$$\pr{^2A_{tot}^\Delta}{u_i\partial u_j} =0.$$
	If~$j$ is adjacent to~$i$, let~$k,l\in V$ be the two opposite vertices in the neighboring triangles~$ijk,ijl\in F_\Delta$. Since
	\begin{align*}
	\frac{\partial A^i_{jk}}{\partial u_i} &= -\frac{1}{2}R_{ijk}^2 \frac{\partial \alpha_i^{jk}}{\partial u_i} = \frac{1}{4} R_{ijk}^2(\cot\alpha_j^{ki} + \cot\alpha_k^{ij}),
	\end{align*}
	the mixed partial derivative equals
	$$
	\pr{^2A_{tot}^\Delta}{u_i\partial u_j} = \underbrace{2A^k_{ij}+2 A^l_{ij}}_{=2A_{ij}} + \frac{1}{2}(R^2_{ijk}\cot\alpha_k^{ij} + R^2_{ijl}\cot\alpha_l^{ij}).
	$$
	Thus,
	\begin{align*}
	d^2A_{tot}^\Delta &= \sum_{ij\in E_\Delta}2A_{ij}(du_i+du_j)^2 - \frac{1}{2}\sum_{ij\in E_\Delta}(R_{ijk}^2\cot\alpha_k^{ij}+R_{ijl}^2\cot\alpha_l^{ij})(du_i-du_j)^2.
	\end{align*}
\end{proof}

\begin{theorem2}[Extension]\label{thm:area_function_extension}
	For a conformal factor~$u\in\R^V$, let~$\Delta(u)$ be a Delaunay triangulation of the surface~$(S,V,d(u))$. The map
	$$A_{tot}:\R^V\to \R, \qquad u\mapsto A_{tot}^{\Delta(u)}(u),$$
	is well-defined and once continuously differentiable.
\end{theorem2}
\begin{proof}
	Due to Lemma~\ref{lemma:area_function_first_derivative} the function~$A_{tot}$ is once continuously differentiable in the interior of every Penner cell. At the boundary between two (or more) Penner cells the triangulations induce the same Delaunay tessellation and thus the same Voronoi tessellation. The areas of the Voronoi cells induced by either of the triangulations are therefore equal. 
\end{proof}

\begin{remark2}
	The function~$A_{tot}$ is, in fact, twice continuously differentiable. This can be proved by a long and unilluminating calculation~\cite[Chapter 8]{ja}.
\end{remark2}

\subsection{The variational principles}\label{sec:three_var_principles}
\begin{theorem2}[Variational principle with equality constraints]\label{thm:variational_principle_E+eq_constraint}
	Let~$(S,V,d)$ be a piecewise flat surface. Up to global rescaling, the PL-metrics with constant discrete Gaussian curvature in the conformal class of the metric~$d$ are in one-to-one correspondence with the critical points of the function $$\E:\R^V\to\R, \qquad u\mapsto \E(u),$$ under the constraint~$$A_{tot}(u)=1.$$
\end{theorem2}
\begin{proof}
	We use the method of Lagrange multipliers. A conformal factor~$u\in\R^V$ is a critical point of the function~$\E$ under the constraint~$A_{tot}=1$ if and only if there exists a Lagrange multiplier~$\lambda\in\R$, such that
	$$0 = \pr{(\E-\lambda A_{tot})}{u_i}\overset{\eqref{eq:E_functional_first_par_derivative}, \eqref{eq:area_functional_first_par_derivative}}{=}W_i-2\lambda A_i.$$
	This holds if and only if
	$$\frac{W_i}{A_i}=2\lambda=const.$$
\end{proof}
The Lagrange multiplier~$\lambda$ satisfies
$$\lambda = \pi\chi(S)$$ by the discrete Gauss-Bonnet theorem.
\begin{theorem*}[Discrete Gauss-Bonnet theorem]
	Let~$(S,V,d)$ be a piecewise flat surface with constant discrete Gaussian curvature~$K_{av}$ at every vertex. Denote the total area of the surface by~$A_{tot}$. Then,
	$$K_{av}= \frac{2\pi\chi(S)}{A_{tot}}.$$
\end{theorem*} 

\begin{fact}[Alternative variational principle to Theorem~\ref{thm:variational_principle_E+eq_constraint}]\label{fact:alternative_vat_principle}
	Up to global rescaling, the PL-metrics with constant discrete Gaussian curvature in the conformal class of the metric~$d$ are in one-to-one correspondence with the critical points of the function $$\F:\R^V\to\R, \qquad u\mapsto \F(u) = \E(u)-\pi\chi(S)\log(A_{tot}(u)).$$
	Indeed, 
	$$0 = \pr{\F}{u_i}\overset{\eqref{eq:E_functional_first_par_derivative}, \eqref{eq:area_functional_first_par_derivative}}{=}W_i-\frac{2\pi\chi(S)}{A_{tot}} A_i.$$
	This holds if and only if
	$$\frac{W_i}{A_i}=\frac{2\pi\chi(S)}{A_{tot}}.$$
\end{fact}
\begin{theorem2}[Variational principle with inequality constraints]\label{thm:variational_principle_E+inequality_constraints}~\\
	Let~$(S,V,d)$ be a piecewise flat surface with~$\chi(S)\neq 0$. The existence of PL-metrics with constant discrete Gaussian curvature in the conformal class of the metric~$d$ follows from the existence of minima of the function~$\E$ under the following inequality constraints:
	\begin{compactitem}
		\item if the Euler characteristic of~$S$ satisfies~$\chi(S)= 2$, the inequality constraint is
		$$A_{tot}\geq 1,$$ 
		\item if the Euler characteristic of~$S$ satisfies~$\chi(S)<0$, the inequality constraint is
		$$A_{tot}\leq 1.$$ 
	\end{compactitem}
\end{theorem2}
\begin{proof}
	Proposition~\ref{prop:max_attained_at_boundary} shows that if~$u\in\R^V$ is a minimum of the function~$\E$ under one of these constraints, then~$A_{tot}(u) = 1$. Since a minimum is a critical point, the claim follows from Theorem~\ref{thm:variational_principle_E+eq_constraint}.
\end{proof}

\begin{proposition2}\label{prop:properties_of_A+-}
	The sets
	$$\mathcal{A}_+ = \{u\in\R^V\mid A_{tot}(u)\geq 1\},\qquad\mathcal{A}_- = \{u\in\R^V\mid A_{tot}(u)\leq 1\},$$
	have the following properties:
	\begin{compactenum}[a)]
		\item 	The sets~$\mathcal{A}_+$ and~$\mathcal{A}_-$ are closed subsets of~$\R^V$.
		\item Let $\1 = (1,\dots,1)\in \R^V$, and let~$u\in \R^V$ be a conformal factor. Then the rays
		$$\mathcal{R}^+_u = \left\{u+c\:\1\mid c\geq -\frac{1}{2}\log A_{tot}(u)\right\},~\mathcal{R}^-_u =\left\{u+c\:\1\mid c\leq -\frac{1}{2}\log A_{tot}(u)\right\}$$
		are completely contained in the sets~$\mathcal{A}_+$ and~$\mathcal{A}_-$, respectively. The sets~$\mathcal{A}_+$ and~$\mathcal{A}_-$ are thus unbounded.
	\end{compactenum}
\end{proposition2}

\begin{proof}
	\begin{compactenum}[a)]
		\item The proof follows from the fact that the sets~$\mathcal{A}_+$ and~$\mathcal{A}_-$ satisfy the equation
		$$\mathcal{A}_+ = A_{tot}^{-1}([1,\infty)),\qquad\mathcal{A}_- = A_{tot}^{-1}([0,1]). $$
		\item The statement follows from the fact that
		$$A_{tot}(u+c\:\1)= A_{tot}(u)\exp(2c).$$
	\end{compactenum}
\end{proof}

\begin{proposition2}\label{prop:max_attained_at_boundary}
	Let~$(S,V,d)$ be a piecewise flat surface. If
	\begin{compactitem}
		\item the Euler characteristic of~$S$ satisfies $\chi(S)=2$ and the function~$\E$ attains a minimum in the set~$\mathcal{A}_+$, or
		\item the Euler characteristic of~$S$ satisfies $\chi(S)<0$ and the function~$\E$ attains a minimum in the set~$\mathcal{A}_-$,
	\end{compactitem} 
	the minimum lies at the boundary of the sets,
	$$\partial\mathcal{A}_+=\partial\mathcal{A}_- = \left\{  u\in\R^V\mid A_{tot}(u)= 1\right\}.$$
\end{proposition2}

\begin{proof}
	Let~$\chi(S)=2$ and let~$u\in \mathcal{A}_+$ be a minimum of the function~$\E$ in~$\mathcal{A}_+$. We show that~$A_{tot}(u)= 1$. Let
	$$c= -\frac{1}{2}\log A_{tot}(u).$$
	Since~$A_{tot}(u)\geq 1$, we know that~$c\leq 0$. Further, $u+c\:\1\in \mathcal{A}_+$ due to Proposition~\ref{prop:properties_of_A+-}. Due to the additive property of the function~$\E$ (Proposition~\ref{prop:properties_E_f}),
	\begin{align*}
	\E(u)\leq \E(u+c\:\1) = \E(u) + 2\chi(S)\pi c \Longrightarrow c\geq 0.
	\end{align*}
	This implies that $c=0$, and thus $ A_{tot}(u)=1$.	\\\\
	For surfaces with~$\chi(S)<0$ the proof is analogous.
\end{proof}

\section{Existence of metrics with constant Gaussian curvature}\label{sec:existence}

In this section we prove Theorem~\ref{thm:main}. We build the proof on several key observations of the behaviour of a sequence~$(u_n)_{n\in\N}$ of conformal factors in~$\R^V$. These observations are central for the application of Theorem~\ref{thm:clas_calculus}, from which the proof of Theorem~\ref{thm:main} follows almost immediately.\\\\
In Section~\ref{sec:pf_thm_main} we reduce the proof of Theorem~\ref{thm:main} to the proofs of Theorem~\ref{thm:divergence_hyperbolic_case} and Theorem~\ref{thm:divergence_spherical_case}. In Section~\ref{sec:existence_technicalities} we study the behaviour of sequences of conformal factors. Finally, in Section~\ref{sec:existence_proofs_thms} we prove Theorem~\ref{thm:divergence_hyperbolic_case} and Theorem~\ref{thm:divergence_spherical_case}.
\subsection{Reduction to Theorem~\ref{thm:divergence_hyperbolic_case} and Theorem~\ref{thm:divergence_spherical_case}}\label{sec:pf_thm_main}

To prove Theorem~\ref{thm:main} we distinguish three cases, corresponding to the three geometries: the spherical case (genus 0, $\chi(S)=2$), the Euclidean case (genus 1, $\chi(S)=0$), and the hyperbolic case (genus $\geq 2$, $\chi(S)<0$).

In the Euclidean case ($\chi(S)=0$) the Yamabe problem is equivalent to the discrete uniformization problem. Theorem~\ref{thm:main} thus follows directly from \cite[Theorem 1.2]{guluo1} and \cite[Theorem 11.1]{bo}.

In the other two cases~($\chi(S)<0$ and~$\chi(S)=2$) finding metrics with constant Gaussian curvature is equivalent to finding the minima of the function~$\E$ in the set~$\mathcal{A}_-$ if~$\chi(S) <0$, and in the set~$\mathcal{A}_+$ if~$\chi(S) = 2$. This is due to Theorem~\ref{thm:variational_principle_E+inequality_constraints}. To prove the existence of these minima we apply Theorem~\ref{thm:clas_calculus} --- a traditional theorem from calculus.
\begin{theorem2}\label{thm:clas_calculus}
	Let~$A\subseteq \R^m$ be a closed set and let~$f:A\to \R$ be a continuous function. If every unbounded sequence~$(x_n)_{n\in \N}$ in~$A$ has a subsequence~$(x_{n_k})_{k\in\N}$ such that $$\lim_{k\to\infty}f(x_{n_k}) = +\infty,$$
	then~$f$ attains a minimum in~$A$.
\end{theorem2}
We already verified that the majority of the conditions of Theorem~\ref{thm:clas_calculus} is satisfied. Proposition~\ref{prop:properties_of_A+-} ensures that the sets~$\mathcal{A}_+$ and~$\mathcal{A}_-$ are closed. Theorem~\ref{thm:E_function_extension} tells us that the function~$\E$ is continuous. To obtain the minima of~$\E$ in the sets~$\mathcal{A}_+$ and~$\mathcal{A}_-$, the following two theorems are left to prove.			\begin{theorem2}\label{thm:divergence_hyperbolic_case}
	Let $\chi(S)<0$ and let $(u_n)_{n\in \N}$ be an unbounded sequence in $\mathcal{A}_-$. Then there exists a subsequence $(u_{n_k})_{k\in\N}$ of $(u_n)_{n\in \N}$, such that
	$$\lim_{k\to\infty}\E(u_{n_k}) = +\infty.$$
\end{theorem2}

\begin{theorem2}\label{thm:divergence_spherical_case}
	Let $\chi(S)=2$ and let $(u_n)_{n\in \N}$ be an unbounded sequence in $\mathcal{A}_+$. Then there exists a subsequence $(u_{n_k})_{k\in\N}$ of $(u_n)_{n\in \N}$, such that
	$$\lim_{k\to\infty}\E(u_{n_k}) = +\infty.$$
\end{theorem2}

\subsection{Behaviour of sequences of conformal factors}\label{sec:existence_technicalities}
Fix a piecewise flat surface~$(S,V,d)$ and let~$(u_n)_{n\in\N}$ be an unbounded sequence in its discrete conformal class~$\R^V$. We denote its coordinate sequence at vertex~$j\in V$ by~$(u_{j,n})_{n\in\N}$.
\begin{tcolorbox}
	\begin{convention}\label{convention}
		Throughout this section we assume that the sequence~$(u_n)_{n\in\N}$ possesses the following properties:
		\\~
		\begin{compactitem}
			\item It lies in one Penner cell~$\mathcal{A}_\Delta$ of~$\R^V$.
			\item There exists a vertex~$i^*\in V$ such that for all~$j\in V$ and~$n\in\N$ $u_{i^*,n}\leq u_{j,n}$.
			\item Each coordinate sequence $(u_{j,n})_{n\in\N}$ either converges, diverges properly to~$+\infty$, or diverges properly to~$-\infty$.
			\item For all~$j\in V$ the sequences~$(u_{j,n}-u_{i^*,n})_{n\in\N}$ either converge or diverge properly to~$+\infty$.
		\end{compactitem}
	\end{convention}
\end{tcolorbox}
We may adopt Convention~\ref{convention} without loss of generality because every sequence in~$\R^V$ possesses a subsequence that satisfies these properties. The first property follows from a theorem by Hirotaka Akiyoshi.
\begin{theorem*}[Hirotaka Akiyoshi~\cite{akiyoshi}]
	The set~$\mathfrak{D}(S,V,d)$ of non-empty Penner cells is finite.
\end{theorem*}
In addition, we use the following notation:
\begin{align} \label{eq:ells}
\ell_{ij}^n:=\ell_{ij}\exp\left(\frac{u_{i,n}+u_{j,n}}{2}\right),
\end{align}
where~$\ell$ is the discrete metric induced by the PL-metric~$d$ on~$(S,V,\Delta)$ (see Fact~\ref{fact:discrete_metrics_induce_PL-metrics}).
Since the sequence~$(u_n)_{n\in\N}$ lies inside the Penner cell~$\mathcal{A}_\Delta$, $\Delta$ is a Delaunay triangulation of~$(S,V,d(u_n))$ for all~$n\in\N$ (see Definition~\ref{def:Penner_sets}). Furthermore, the map~$\ell^n$ defined by Equation~\eqref{eq:ells} is the discrete metric induced on~$(S,V,\Delta)$ by the PL-metric~$d(u_n)$ (see Proposition~\ref{prop:equivalence_disc_conformal_fixed_triangulation}).

\paragraph{Behaviour of~$\mathbf{(u_n)_{n\in\N}}$ in one triangle}
Consider a triangle in~$F_\Delta$ with vertices labeled by~$1,2,3\in V$ and initial edge lengths~$\ell_{12},\ell_{23},\ell_{31}$, uniquely determined by~$d$. Define
\begin{align}\label{eq:A_123}
\mathcal{A}_{123} := \{ (u_1,u_2,u_3) \mid u\in \mathcal{A}_\Delta \}.
\end{align}
\begin{figure}[h!]
	\labellist
	\small\hair 2pt
	\pinlabel {$1$} [ ] at 48 81
	\pinlabel {$2$} [ ] at 532 21
	\pinlabel {$3$} [ ] at 329 363
	\pinlabel {$\alpha_n$} [ ] at 143 115
	\pinlabel {$\ell_{12}^n $} [ ] at 281 20
	\pinlabel {$\ell_{31}^n$} [ ] at 160 261
	\pinlabel {$\ell_{23}^n$} [ ] at 430 220
	\endlabellist
	\centering
	\includegraphics[width=0.3\textwidth]{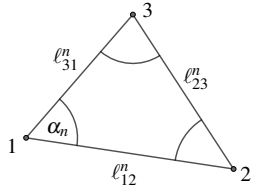}
	\caption{Sequences of edge lengths on triangle 123. }
	\label{fig:f_function_triangle1}
\end{figure}

Let~$(u_{1,n},u_{2,n},u_{3,n})_{n\in\N}$ be a sequence in~$\mathcal{A}_{123}$. Then the edge lengths~$\ell_{12}^n, \ell_{23}^n,\ell_{31}^n$ satisfy the triangle inequalities for all~$n\in\N$.
\begin{lemma2}\label{lemma:triangle_+infty}
	If~$u_{1,n}\xrightarrow{n\to\infty}\infty, u_{2,n}\xrightarrow{n\to\infty}\infty$ and the sequence~$(u_{3,n})_{n\in\N}$ is bounded from above,
	there exists an~$n\in \N$ such that
	$$\ell_{12}^n> \ell_{23}^n+\ell_{31}^n.$$
	In other words, there exists \textbf{\underline{no}} sequence in~$\mathcal{A}_{123}$ where two of the coordinate sequences would diverge properly to~$+\infty$ and the third one would be bounded from above.
\end{lemma2}

\begin{proof}
	Without loss of generality we may assume that~$u_{1,n} \leq u_{2,n}$ for all~$n\in \N$. Then
	\begin{align*}
	0&<\ell_{12} = \exp\left( \frac{-u_{1,{n}}-u_{2,{n}}}{2} \right)\ell_{12}^{n}\\
	&\overset{\Delta\text{-ineq.}}{\leq} \exp\left( \frac{-u_{1,{n}}-u_{2,{n}}}{2}\right)(\ell_{23}^{n}+\ell_{31}^{n})\\
	&=\exp\left( \frac{u_{3,{n}}-u_{1,{n}}}{2}\right)\left(\ell_{23}+\ell_{31}\underbrace{\exp\left( \frac{u_{1,{n}}-u_{2,{n}}}{2}\right)}_{\leq 1}\right)\\
	&\leq \exp\left(\frac{u_{3,{n}}-u_{1,{n}}}{2}\right)(\ell_{23}+\ell_{31})\xrightarrow{{n} \to \infty}0.
	\end{align*}
	This contradicts the triangle inequality
	$$\ell_{12}^n\leq \ell_{23}^n+\ell_{31}^n.$$
\end{proof}

We now make a subtle shift of perspective --- instead of studying the development of triangles under sequences of conformal factors~$(u_{i,n})_{n\in\N}$, we consider their development under the sequences~$(u_{i,n}-u_{i^*,n})_{n\in\N}$. Geometrically, this corresponds to the rescaling of the whole triangulation by a factor~$\exp{(-u_{i^*,n})}$ at each step~$n$.

Since we are primarily interested in the conditions under which the triangle inequalities break (such as those in Lemma~\ref{lemma:triangle_+infty}), this shift is an elegant way to reduce the number of cases. Indeed, a triangle with conformal factors~$(2n, 2n, n)_{n\in\N}$ will degenerate just as the triangle with conformal factors~$(n, n, 0)_{n\in\N}$ will, since the triangles are similar.

Lemma~\ref{lemma:triangle_+infty} yields the following key observation:
\begin{corollary2}\label{cor:two_vertices_converge}
	At every triangle~$ijk\in F_\Delta$, at least two of the three sequences~$(u_{i,n}-u_{i^*,n})_{n\in\N}$,~$(u_{j,n}-u_{i^*,n})_{n\in\N}$,~$(u_{k,n}-u_{i^*,n})_{n\in\N}$ converge.
\end{corollary2}
\begin{proof}
	The claim holds for any triangle with vertex~$i^*$ due to Lemma~\ref{lemma:triangle_+infty}. It holds for all remaining triangles in~$ F_\Delta$ due to the connectivity of the triangulation. 
\end{proof}
\begin{lemma2}\label{lemma:triangle_one_vertex_+infty}
	Assume that the sequence~$(u_{1,{n}})_{n\in\N}$ diverges properly to~$+\infty$ and the sequences~$(u_{2,{n}})_{n\in\N}$ and~$(u_{3,{n}})_{n\in\N}$ converge.
	Then 	$$\frac{\ell_{12}^{n}}{\ell_{31}^{n}}\xrightarrow{n\to\infty} 1,$$
	and the sequence of angles~$\alpha_n$, opposite to the edge~$23$ in the triangle with edge lengths~$\ell_{12}^n,\ell_{23}^n,\ell_{31}^n$, satisfies
	$$\alpha_n\xrightarrow{n\to\infty} 0.$$
\end{lemma2}

\begin{proof}
	Dividing both sides of the triangle inequality~$\ell_{31}^n \leq \ell_{23}^n + \ell_{12}^n$ by~$\ell_{31}^n$ yields the inequality
	\begin{align*}
	&1 \leq \frac{\ell_{23}^n}{\ell_{31}^n} + \frac{\ell_{12}^n}{\ell_{31}^n} = \frac{\ell_{23}}{\ell_{31}} \exp\left(\frac{1}{2}(u_{2,n}-u_{1,n})\right)+
	\frac{\ell_{12}^n}{\ell_{31}^n}.
	\end{align*}
	Dividing both sides of the triangle inequality~$\ell_{12}^n \leq \ell_{23}^n + \ell_{31}^n$ by~$\ell_{12}^n$ yields the inequality
	\begin{align*}
	&1 \leq \frac{\ell_{23}}{\ell_{12}} \exp\left(\frac{1}{2}(u_{3,n}-u_{1,n})\right)+
	\frac{\ell_{31}^n}{\ell_{12}^n}.
	\end{align*}
	Since, for~$i=2,3,
	\exp\left(\frac{1}{2}(u_{i,n}-u_{1,n})\right)\xrightarrow[]{n\to\infty} 0,$
	we obtain
	$$\frac{\ell_{23}^{n}}{\ell_{31}^{n}}\xrightarrow[]{n\to\infty} 0,\qquad \frac{\ell_{23}^{n}}{\ell_{12}^{n}}\xrightarrow[]{n\to\infty} 0.$$
	The convergence of the fraction~$\frac{\ell_{12}^n}{\ell_{31}^n}$ follows from the inequalities
	$$1 \leq \lim_{n\to\infty}\frac{\ell_{12}^n}{\ell_{31}^n}\leq 1.$$
	From the cosine rule we obtain the convergence
	\begin{align*}
	2\cos\alpha_n = \frac{\ell_{12}^n}{\ell_{31}^n} + \frac{\ell_{31}^n}{\ell_{12}^n} - \frac{(\ell_{23}^n)^2}{\ell_{31}^n\ell_{12}^n}\xrightarrow[]{n\to\infty}2,
	\end{align*}
	and thus $\alpha_n \xrightarrow[]{n\to\infty} 0.$
\end{proof}

\paragraph{Behaviour of~$\mathbf{(u_n)_{n\in\N}}$ around a vertex star}
Let~$i\in V$ be a vertex such that the sequence~$(u_{i,n})_{n\in\N}$ diverges properly to~$+\infty$ and the sequences~$(u_{j,n})_{n\in\N}$ at any neighbour~$j\in V$ converge. We investigate the behaviour of angles in triangles with vertex~$i$.

\begin{figure}[htb]
	\labellist
	\small\hair 2pt
	\pinlabel {$0$} [ ] at 278 330
	\pinlabel {$1$} [ ] at 423 497
	\pinlabel {$j$} [ ] at 161 0
	\pinlabel {$j+1$} [ ] at 462 22
	\pinlabel {$2$} [ ] at 179 442
	\pinlabel {$s$} [ ] at 602 306

	\pinlabel {$\alpha_j^{0,j+1}$} [ ] at 208 37
	\pinlabel {$\alpha_{j+1}^{0,j}$} [ ] at 395 73
	\pinlabel {$\alpha_0^{j,j+1}$} [ ] at 286 214
	\pinlabel {$\ell_{0,j}^n$} [ ] at 178 159
	\pinlabel {$\ell_{0,j+1}^n$} [ ] at 402 195
	\endlabellist
	\centering
	\includegraphics[width=0.6\textwidth]{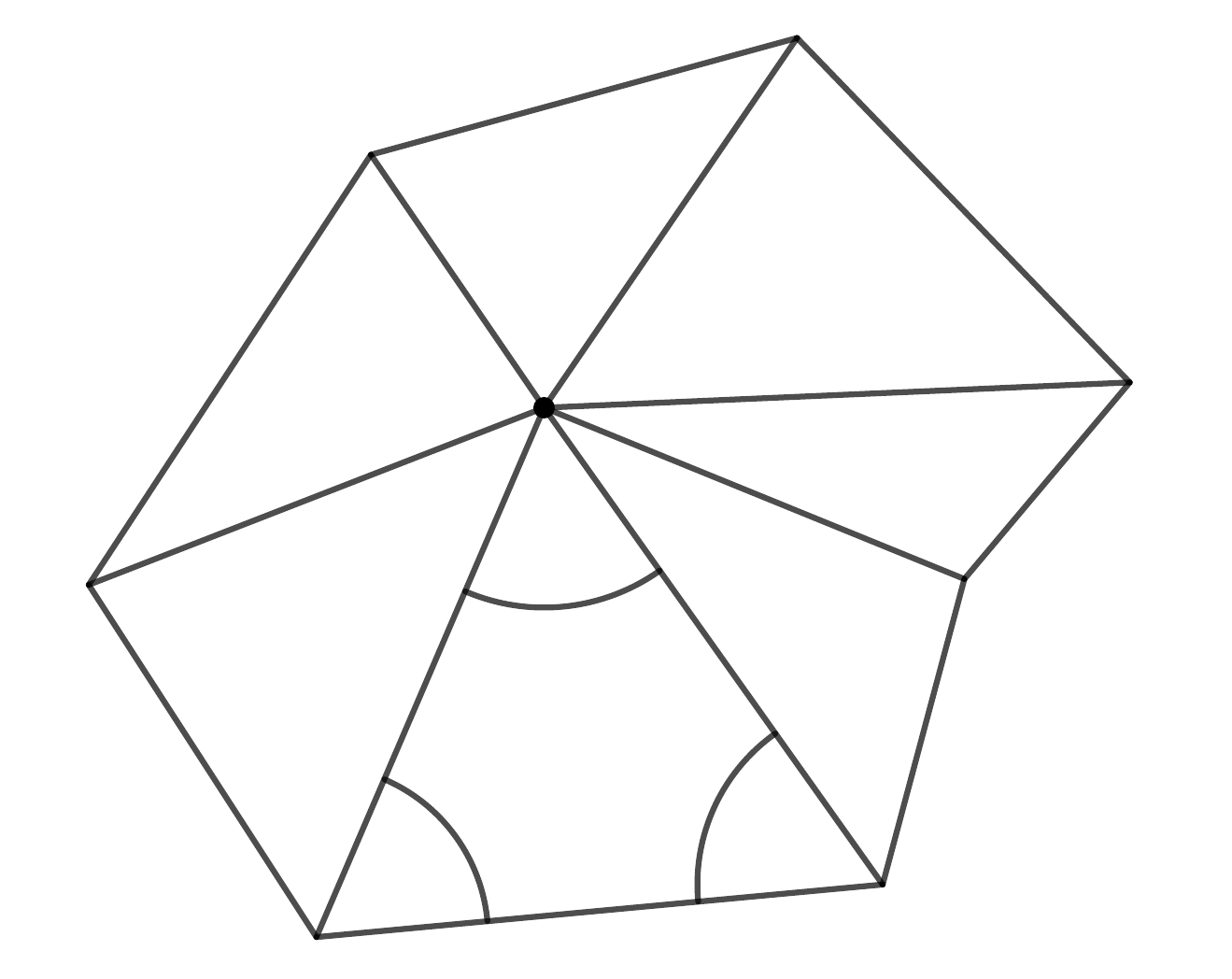}
	\caption{ Labeling in a vertex star. }
	\label{fig:angles_vertex_star}
\end{figure}
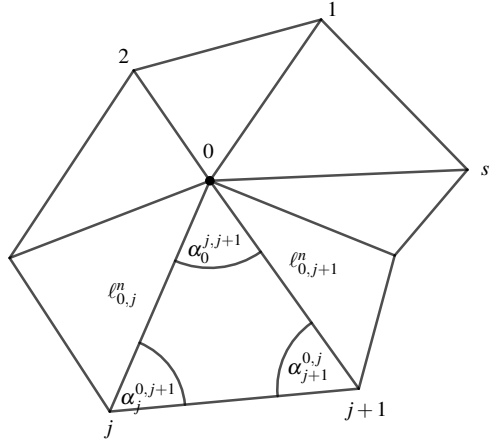

A \textbf{vertex star} around vertex~$i$ is the subset~$F_\Delta^{i}\subseteq 
F_\Delta$ of triangles in~$F_\Delta$ that contain the vertex~$i$. We denote the degree of the vertex $i$ by~$s$ and label the vertices as in Figure~\ref{fig:angles_vertex_star}. We drop the index~$n$ that denotes the elements in the sequence when we label angles. Whenever the labeling requires it we use the conventions~$1-1 = s$ and $s+1 = 1$.

\begin{proposition2}\label{prop:area_vertex_star_bounded}
	The sequences of angles in the triangles of~$F_\Delta^i$ satisfy
	$$\lim_{n\to \infty}\alpha^{j,j+1}_0 = 0,\qquad \lim_{n\to \infty}\alpha^{0,j}_{j+1} = \lim_{n\to \infty}\alpha^{0,j+1}_j =\pi/2, \qquad j\in\{1,\dots,s\}.$$
\end{proposition2}
\begin{proof}
	Denote the limit of a sequence of angles $\alpha^{i,j}_k$ along $(u_n)_{n\in\N}$ by $\bar{\alpha}^{i,j}_k$. Due to Lemma~\ref{lemma:triangle_one_vertex_+infty}, 
	$$\bar{\alpha}_0^{j,j+1} = 0,$$
	and thus, for all~$j=1,\dots,s$,
	\begin{align}\label{eq:angle_sum}
	\bar{\alpha}_j^{0,j+1}+ \bar{\alpha}_{j+1}^{0,j}= \pi.
	\end{align}
	
	Since the edges $0j$ are Delaunay, the Delaunay inequality
	\begin{align}\label{eq:Delaunay_ineq_eq}
	\bar{\alpha}_{j-1}^{0,j}+ \bar{\alpha}_{j+1}^{0,j}\leq \pi
	\end{align}
	is satisfied for each~$j \in\{1,\dots,s\}$. 
	Summing up the Delaunay inequalities we obtain
	\begin{align*}
	\pi s\overset{\eqref{eq:Delaunay_ineq_eq}}{\geq} \sum_{j=1}^s (\bar{\alpha}_{j-1}^{0,j}+ \bar{\alpha}_{j+1}^{0,j}) = \sum_{j=1}^s (\bar{\alpha}_{j}^{0,j+1}+ \bar{\alpha}_{j+1}^{0,j}) \overset{\eqref{eq:angle_sum}}{=} \pi s.
	\end{align*}
	In other words, each Delaunay inequality~\eqref{eq:Delaunay_ineq_eq} becomes an \underline{equality} in the limit. Due to equality~\eqref{eq:angle_sum},
	\begin{align*}
	\bar{\alpha}_{j-1}^{0,j} = \bar{\alpha}_{1}^{0,2}, \qquad \bar{\alpha}_{j}^{0,j-1} = \pi-\bar{\alpha}_{1}^{0,2},
	\end{align*}
	for all $j\in\{1,\dots,s\}$.\\\\
	To show that~$\bar{\alpha}_{1}^{0,2}=\pi/2$, we apply the following equation:\\\\
	In a triangle with sides~$a,b,c$, and opposite angles~$\alpha,\beta,\gamma$,
	\begin{align}\label{eq:edge_difference}
	b-a = c\frac{\sin\left(\frac{\alpha-\beta}{2}\right)}{\cos\left(\frac{\gamma}{2}\right)}.
	\end{align}
	
	\par
	Denote the limit of the lengths of edges~$\ell_{j,j+1}^n$ by~$\lim_{n\to\infty}\ell_{j,j+1}^n = \bar{\ell}_{j,j+1}.$
	Since, for all~$n\in\N$, holds
	$$\sum_{j=1}^s \left(\ell_{0,j+1}^n - \ell_{0,j}^n\right) = 0,$$
	in the limit
	\begin{align*}
	0=\lim_{n\to\infty}\sum_{j=1}^s \left(\ell_{0,j+1}^n - \ell_{0,j}^n\right) \overset{\eqref{eq:edge_difference}}{=} \sin\left(\frac{\pi-2\bar{\alpha}_{1}^{0,2}}{2}\right)\sum_{j=1}^s \bar{\ell}_{j,j+1}.
	\end{align*}
	Since, for all~$j=1,\dots,s$, the sequences of conformal factors~$(u_{j,n})_{n\in\N}$ converge,
	$$\sum_{j=1}^s \bar{\ell}_{j,j+1}>0.$$
	We deduce that
	$$\sin\left(\frac{\pi-2\bar{\alpha}_{1}^{0,2}}{2}\right)=0,$$
	and thus $\bar{\alpha}_{1}^{0,2} = \pi/2.$	
\end{proof}

\paragraph{Behaviour of the function~$\E$ along~$\mathbf{(u_n)_{n\in\N}}$}\label{sec:existence_behaviour_of_E_along_sequences}

Recall the definitions of the function~$f$ (Definition~\ref{def:the_function_f}) and the set~$\mathcal{A}_{123}$ (Equation~\eqref{eq:A_123}). Let
$$h:\mathcal{A}_{123}\to \R,\qquad h(u_1,u_2,u_3) := 2f\left( \frac{\tilde{\lambda}_{12}}{2},\frac{\tilde{\lambda}_{23}}{2},\frac{\tilde{\lambda}_{31}}{2} \right) - \frac{\pi}{2}(\tilde{\lambda}_{12}+\tilde{\lambda}_{23}+\tilde{\lambda}_{31}).$$

\begin{lemma2}\label{lemma:h-scaling} For any real number~$v\in \R$, the function~$h$ satisfies the equation
	$$h((u_1,u_2, u_3) + v(1,1,1)) = h(u_1,u_2,u_3)-\pi v.$$
\end{lemma2}
\begin{proof}
	Follows from the property of the function $f$ from Proposition~\ref{prop:properties_E_f}.
\end{proof}

\begin{proposition2}\label{prop:behaviour_of_h}
	Let~$(u_{1,n},u_{2,n},u_{3,n})_{n\in \N}$ be a sequence in~$\mathcal{A}_{123}$. Suppose that 
	$$u_{1,n}\xrightarrow{n\to\infty}+\infty, \qquad u_{2,n}\xrightarrow{n\to\infty} \overline{u_2},\qquad u_{3,n}\xrightarrow{n\to\infty} \overline{u_3}.$$
	Then the sequence~$(h(u_{1,n},u_{2,n},u_{3,n}))_{n\in \N}$ converges, and in particular
	$$
	\lim_{n\to\infty} h(u_{1,n},u_{2,n},u_{3,n}) =-\pi \left(\log\ell_{23}+\frac{1}{2}(\overline{u_2}+ \overline{u_3})\right).
	$$
\end{proposition2}
\begin{proof}
	\begin{figure}[h!]
		\labellist
		\small\hair 2pt
		\pinlabel {$1$} [ ] at 48 81
		\pinlabel {$2$} [ ] at 532 21
		\pinlabel {$3$} [ ] at 329 363
		\pinlabel {$\alpha_n$} [ ] at 133 120
		\pinlabel {$\beta_n$} [ ] at 433 79
		\pinlabel {$\gamma_n$} [ ] at 291 295
		\pinlabel {\rotatebox{352}{$\ell_{12}^n = \exp(z_n)$}} [ ] at 261 40
		\pinlabel {\rotatebox{50}{$\ell_{31}^n = \exp(y_n)$}} [ ] at 180 251
		\pinlabel {\rotatebox[]{305}{$\ell_{23}^n = \exp(x_n)$}} [ ] at 420 220
		\endlabellist
		\centering
		\includegraphics[width=0.4\textwidth]{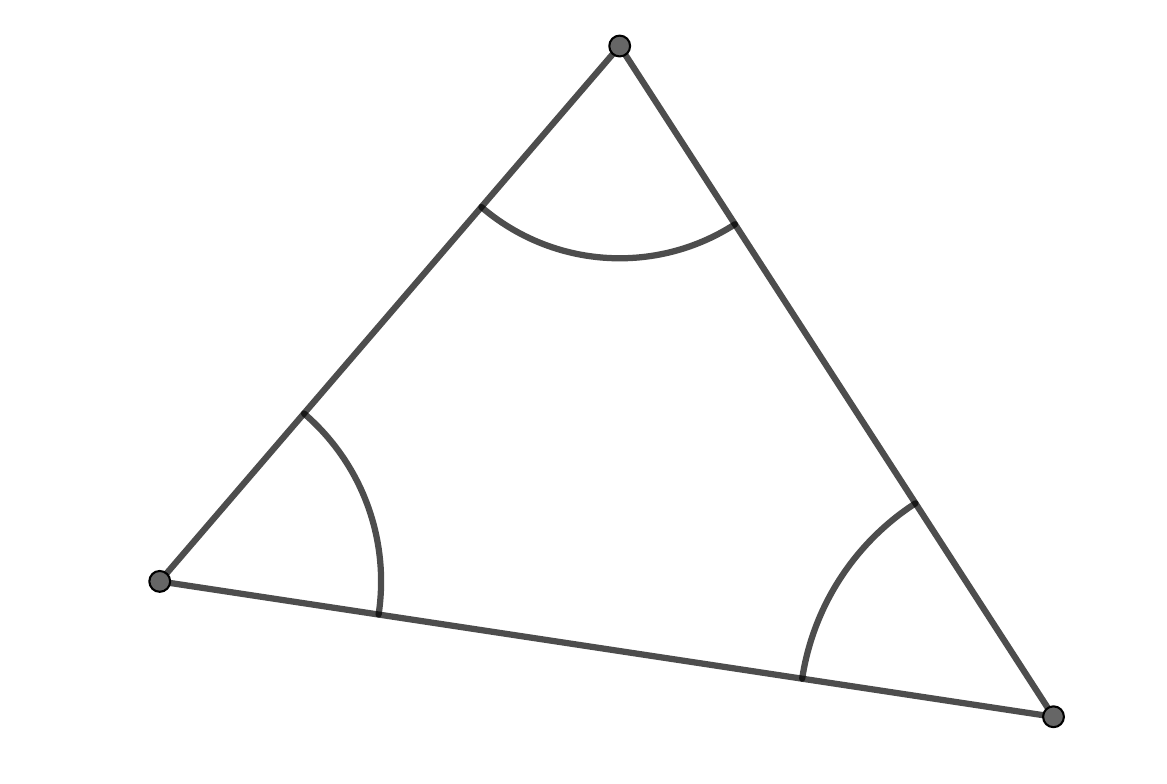}
		\caption{ }
		\label{fig:f_function_triangle}
	\end{figure}
	Consider the notation as in Figure~\ref{fig:f_function_triangle}.
	Then,
	\begin{align*}
	\frac{1}{2}h(u_{1,n},u_{2,n},u_{3,n}) = &\alpha_n x_n + \beta_n y_n + \gamma_n z_n + \LobL(\alpha_n) + \LobL(\beta_n)+\LobL(\gamma_n)\\
	&\qquad- \frac{\pi}{2}(x_n+y_n+z_n).
	\end{align*}
	In the limit, the sequences~$(x_n)_{n\in\N},(y_n)_{n\in\N}$ and~$(z_n)_{n\in\N}$ of edge lengths satisfy		
	$$\lim_{n\to\infty} x_n =  \log\ell_{23}+\frac{1}{2}(\overline{u_2}+ \overline{u_3})=:\overline{x},\qquad \lim_{n\to\infty} y_n = +\infty, \qquad \lim_{n\to\infty} z_n = +\infty,$$
	and, due to Proposition~\ref{prop:area_vertex_star_bounded},
	$$\lim_{n\to\infty}(\alpha_n,\beta_n, \gamma_n) = \left( 0,\frac{\pi}{2}, \frac{\pi}{2}\right).$$
	Thus,
	$$\lim_{n\to\infty}\alpha_n x_n =0,$$
	and, since the Lobachevsky function is continuous and satisfies the equality~$\LobL(0)=\LobL\left(\frac{\pi}{2}\right)= 0$
	(see Fact~\ref{fact:lobachevsky_function}), in the limit we obtain
	$$\lim_{n\to\infty}(\LobL(\alpha_n) + \LobL(\beta_n)+\LobL(\gamma_n))=0.$$
	In summary,
	\begin{align*}
	\lim_{n\to\infty}h(u_{1,n},u_{2,n},u_{3,n}) = 2\lim_{n\to\infty} \left[\left(\beta_n - \frac{\pi}{2}\right)y_n+\left(\gamma_n - \frac{\pi}{2}\right)z_n  \right] - \pi \overline{x}.
	\end{align*}
	We rearrange the expression $\left(\beta_n - \frac{\pi}{2}\right)y_n+\left(\gamma_n - \frac{\pi}{2}\right)z_n$ to obtain
	$$\left(\beta_n - \frac{\pi}{2}\right)y_n+\left(\gamma_n - \frac{\pi}{2}\right)z_n = -\frac{1}{2}\alpha_n(y_n+z_n) + \frac{1}{2}(\beta_n-\gamma_n)(y_n-z_n).$$
	In the limit, $\lim_{n\to\infty}(\beta_n-\gamma_n) = 0$ due to Proposition~\ref{prop:area_vertex_star_bounded},
	and
	\begin{align*}
	\lim_{n\to\infty}(y_n-z_n) &=\log\ell_{31}-\log\ell_{12}+\frac{1}{2}(\overline{u_3}-\overline{u_2}).
	\end{align*}
	Thus,
	\begin{align*}
	\lim_{n\to\infty}\frac{1}{2}(\beta_n-\gamma_n)(y_n-z_n)=0.
	\end{align*}
	It is left to determine the limit
	\begin{align*}
	\lim _{n\to\infty}\alpha_n(y_n+z_n) &=\lim _{n\to\infty}\alpha_n\log \ell_{31}^n + \lim _{n\to\infty}\alpha_n \log \ell_{12}^n.
	\end{align*}
	We recall that due to Proposition~\ref{prop:area_vertex_star_bounded}
	$ \lim _{n\to\infty}\alpha_n = 0,$ and that $ \lim _{n\to\infty}\log \ell_{31}^n = \lim _{n\to\infty}\log \ell_{12}^n = +\infty.$
	We apply the sine rule and the L'Hospital's rule to obtain the expression
	\begin{align*}
	\lim _{n\to\infty}\alpha_n\log \ell_{31}^n &= \lim _{n\to\infty}({\alpha_n}{\log \ell_{23}^n} + \alpha_n{\log \sin\beta_n} - \alpha_n\log\sin\alpha_n)\\
	&=-\lim_{n\to\infty}\alpha_n\log\sin\alpha_n = 0.
	\end{align*}
	Similarly, $\lim _{n\to\infty}\alpha_n\log \ell_{12}^n = 0.$\\\\
	Altogether, we see that
	$$\lim_{n\to\infty}h(u_{1,n},u_{2,n},u_{3,n}) =  -\pi \overline{x}.$$
\end{proof}

\begin{lemma2}\label{lemma:E_divergence}
	There exists a convergent sequence~$(D_n)_{n\in\N}$ of real numbers such that the function~$\E$ satisfies
	\begin{align*}
	\E(u_{n}) =D_{n} + 2\pi \left(u_{i^*, n} \chi(S) + \sum_{j\in V}(u_{j,n}-u_{i^*, n}) \right).
	\end{align*}
\end{lemma2}
\begin{proof}
	Due to the Euler formula, $ 2\abs{V} - \abs{F_\Delta} = 2\chi(S).$ Applying Lemma~\ref{lemma:h-scaling} we obtain the equality
	\begin{align*}
	\E(u_{n}) &=\sum_{ijl\in F_\Delta} h(u_{i,n},u_{j,n},u_{l,n})+ 2\pi \sum_{j\in V}u_{j,n}\\
	&=\underbrace{\sum_{ijl\in F_\Delta} h((u_{i,n},u_{j,n},u_{l,n})-u_{i^*, n}(1,1,1)) }_{=:D_{n}}- \pi \abs{F_\Delta}u_{i^*, n}+ 2\pi \sum_{j\in V}u_{j,n}\\
	&=D_{n} + 2\pi \left(u_{i^*, n} \chi(S) + \sum_{j\in V}(u_{j,n}-u_{i^*, n}) \right).
	\end{align*}
	
	The sequence~$(D_{n})_{n\in\N}$ converges due to Corollary~\ref{cor:two_vertices_converge} and Proposition~\ref{prop:behaviour_of_h}.
\end{proof}
\paragraph{Influence of~$\mathbf{(u_n)_{n\in\N}}$ on the area of a triangle}
\begin{lemma2}\label{lemma:triangulation_triangle_areas}
	Let~$ijk\in F_\Delta$ such that the sequences~$(u_{j,{n}}-u_{i^*,{n}})_{n\in\N}$ and ~$(u_{k,{n}}-u_{i^*,{n}})_{n\in\N}$ converge. Denote by~$A_{ijk}^n$ the area of the triangle with edge lengths~$\ell_{ij}^n,\ell_{jk}^n,\ell_{ki}^n.$
	
	\begin{enumerate}[a)]
		\item If the sequence~$(u_{i,{n}}-u_{i^*,{n}})_{n\in\N}$ converges, there exists a convergent sequence of real numbers~$(C_n)_{n\in\N}$, such that the area of the triangle with edge lengths~$\ell_{ij}^n,\ell_{jk}^n,\ell_{ki}^n$ satisfies
		$$\log A_{ijk}^n= C_n + 2 u_{i^*,{n}}.$$
		\item If the sequence~$(u_{i,{n}}-u_{i^*,{n}})_{n\in\N}$ diverges to~$+\infty$, there exists a convergent sequence of real numbers~$(C_n)_{n\in\N}$, such that the area of the triangle with edge lengths~$\ell_{ij}^n,\ell_{jk}^n,\ell_{ki}^n$ satisfies
		$$ \log A_{ijk}^n= C_n + \frac{1}{2} (u_{i,{n}}+3u_{i^*,{n}}).$$
	\end{enumerate}
\end{lemma2}

\begin{proof}
	
	\begin{figure}[h!]
		\labellist
		\small\hair 2pt
		\pinlabel {$i$} [ ] at 8 71
		\pinlabel {$j$} [ ] at 532 21
		\pinlabel {$l$} [ ] at 279 413
		\pinlabel {$\alpha_l^n$} [ ] at 260 325
		\pinlabel {$\ell_{ij}^n $} [ ] at 261 30
		\pinlabel {$\ell_{il}^n$} [ ] at 90 251
		\pinlabel {$\ell_{jl}^n$} [ ] at 420 220
		\endlabellist
		\centering
		\includegraphics[width=0.3\textwidth]{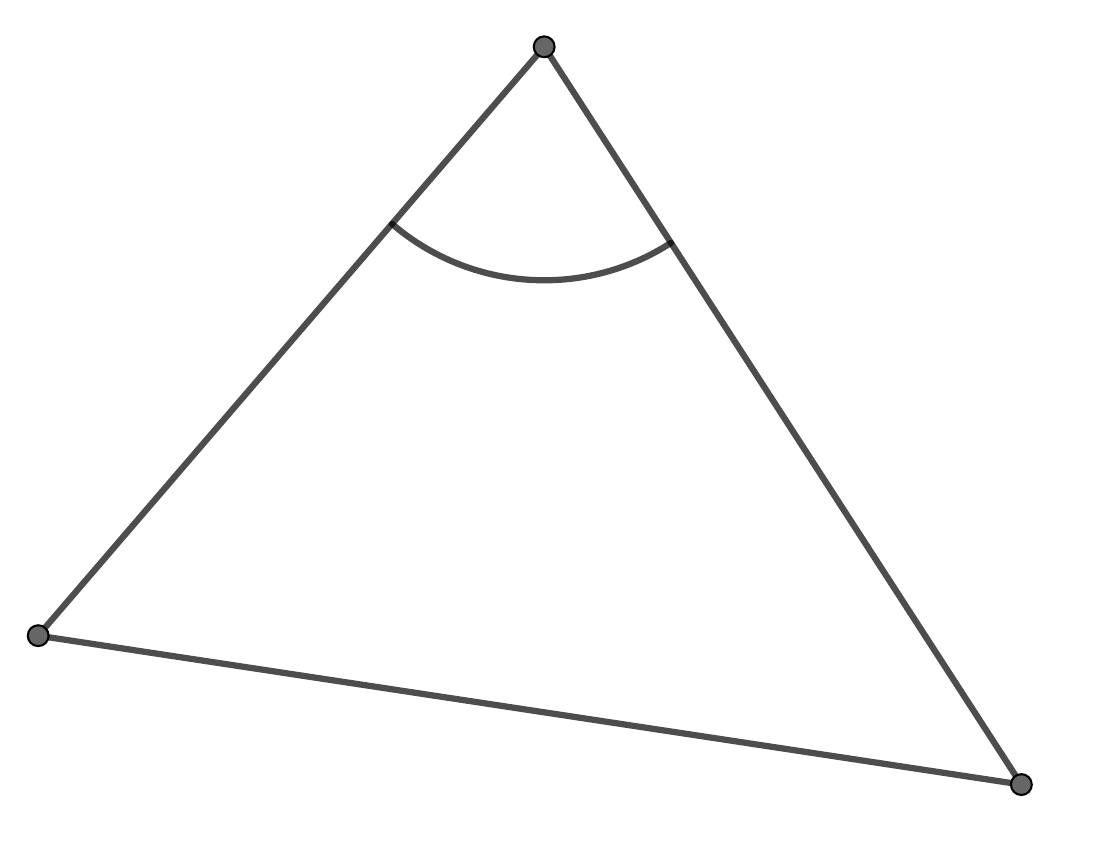}
		\caption{ }
		\label{fig:triangle_area_logarithm}
	\end{figure}
 The proof follows from the continuity of the area function, from Convention~\ref{convention} and from Corollary~\ref{cor:two_vertices_converge}. Indeed, let~$\alpha^n_l$ be the angle at vertex~$l$, as in Figure~\ref{fig:triangle_area_logarithm}. Then
 \begin{multline*}
 \log A^n_{ijl} = \underbrace{\log\left(\frac{1}{2}\ell_{il}\ell_{jl}\sin\alpha^n_l\right) + u_{l,{n}}-u_{i^*,{n}}+ \frac{1}{2}(u_{j,{n}}-u_{i^*,{n}})}_{(\ast)}\\+\frac{1}{2}(u_{i,{n}}-u_{i^*,{n}})+2u_{i^*,{n}},
 \end{multline*}
	where~$(\ast)$ converges due to the assumption and due to Proposition~\ref{prop:area_vertex_star_bounded}. If the sequence~$(u_{i,{n}}-u_{i^*,{n}})_{n\in\N}$ converges,
	define~$C_n = (\ast) + \frac{1}{2}(u_{i,{n}}-u_{i^*,{n}})$. If the sequence~$(u_{i,{n}}-u_{i^*,{n}})_{n\in\N}$ diverges to~$+\infty$, define~$C_n = (\ast)$. In both cases the sequence~$(C_n)_{n\in\N}$ converges, and the result follows.
\end{proof}

\subsection{Proofs of Theorem~\ref{thm:divergence_hyperbolic_case} and Theorem~\ref{thm:divergence_spherical_case}}\label{sec:existence_proofs_thms}

\begin{theorem*}[Theorem~\ref{thm:divergence_hyperbolic_case}]
	Let $\chi(S)<0$ and let $(u_n)_{n\in \N}$ be an unbounded sequence in $\mathcal{A}_-$. Then there exists a subsequence $(u_{n_k})_{k\in\N}$ of $(u_n)_{n\in \N}$, such that
	$$\lim_{k\to\infty}\E(u_{n_k}) = +\infty.$$
\end{theorem*}

\begin{proof}
	We assume that the sequence~$(u_{n})_{n\in\N}$ satisfies Convention~\ref{convention}. Due to Lemma~\ref{lemma:E_divergence} there exists a convergent sequence~$(C_n)_{n\in\N}$ such that
	\begin{align*}
	\E(u_{n}) =C_{n} + 2\pi \left(u_{i^*, n} \chi(S) + \sum_{j\in V}(u_{j,n}-u_{i^*, n}) \right).
	\end{align*}
	The sequence
	$$\left(\sum_{j\in V}(u_{j,n}-u_{i^*, n})\right)_{n\in\N}$$ is bounded from below by zero due to Convention~\ref{convention}.
	\\\\
	Since the sequence~$(u_{n})_{n\in\N}$ lies in~$\mathcal{A}_-$, the area of each triangle is bounded from above. At the same time~$(u_{n})_{n\in\N}$ is unbounded. We apply Lemma~\ref{lemma:triangulation_triangle_areas} to conclude that the sequence~$(u_{i^*,n})_{n\in\N}$ diverges properly to~$-\infty$.
	
	Indeed, if~$(u_{i^*,n})_{n\in\N}$ would diverge properly to~$+\infty$, any of the two cases of Lemma~\ref{lemma:triangulation_triangle_areas} would yield a contradiction to the bound on the area of any triangle. Assume that~$(u_{i^*,n})_{n\in\N}$ converges. If all triangles satisfy the condition of case $a)$ of Lemma~\ref{lemma:triangulation_triangle_areas} then all sequences~$(u_{i,n})_{n\in\N}$ converge --- a contradiction to the fact that~$(u_{n})_{n\in\N}$ is unbounded. Thus there must be one sequence~$(u_{i,n})_{n\in\N}$ such that~$(u_{i,{n}}-u_{i^*,{n}})_{n\in\N}$ diverges to~$+\infty$. This in turn implies that~$(u_{i,n})_{n\in\N}$ itself diverges to~$+\infty$. Applying case~$b)$ of Lemma~\ref{lemma:triangulation_triangle_areas} yields the contradiction to the upper bound on the area of any triangle with vertex~$i$.
	
	Thus, the sequence~$(u_{i^*,n})_{n\in\N}$ must diverge properly to~$-\infty$, and
	$$\lim_{n\to\infty}\E(u_{n}) = +\infty.$$
\end{proof}
\begin{theorem*}[Theorem~\ref{thm:divergence_spherical_case}]
	Let~$\chi(S)=2$ and let~$(u_n)_{n\in \N}$ be an unbounded sequence in~$\mathcal{A}_+$. Then there exists a subsequence $(u_{n_k})_{k\in\N}$ of~$(u_n)_{n\in \N}$, such that
	$$\lim_{k\to\infty}\E(u_{n_k}) = +\infty.$$
\end{theorem*}
\begin{proof}
	We assume that the sequence~$(u_{n})_{n\in\N}$ satisfies Convention~\ref{convention}. Due to Lemma~\ref{lemma:E_divergence} there exists a convergent sequence~$(C_n)_{n\in\N}$ such that
	\begin{align*}
	\E(u_{n}) =C_{n} + 2\pi \left( 2 u_{i^*, n} + \sum_{j\in V}(u_{j,n}-u_{i^*, n}) \right).
	\end{align*}
	The sequence
	$$\left(\sum_{j\in V}(u_{j,n}-u_{i^*, n})\right)_{n\in\N}$$ is bounded from below by zero due to Convention~\ref{convention}. We distinguish three cases.\\\\
	\textbf{Case 1: The sequence}~$\mathbf{(u_{i^*, n})_{n\in\N}}$ \textbf{diverges properly to}~$\mathbf{+\infty}.$\\
	It follows immediately that
	$$\lim_{n\to\infty}\E(u_{n}) = +\infty.$$
	\textbf{Case 2: The sequence}~$\mathbf{(u_{i^*, n})_{n\in\N}}$ \textbf{converges.}\\
	Since the sequence~$(u_n)_{n\in \N}$ is unbounded, there exists a vertex~$j\in V$ with $\lim_{n\to\infty}(u_{j,n}-u_{i^*, n})=+\infty$. Thus,
	$$\lim_{n\to\infty}\E(u_{n}) = +\infty.$$
	\textbf{Case 3: The sequence}~$\mathbf{(u_{i^*, n})_{n\in\N}}$ \textbf{diverges properly to}~$\mathbf{-\infty}.$\\
	There exists a vertex~$i\in V$, such that the sequence~$(u_{i,{n}}+3u_{i^*,{n}})_{n\in\N}$ is bounded from below. This is due to the fact that the sequence~$(u_{n})_{n\in\N}$ lies in~$\mathcal{A}_+$, and thus there exists a triangle whose area is non-zero in the limit. The lower bound then follows from Lemma~\ref{lemma:triangulation_triangle_areas} case~$b)$. We obtain
	\begin{align*}
	2u_{i^*, n} + \sum_{j\in V}(u_{j,n}-u_{i^*, n}) &= -2u_{i^*, n}+ (u_{i,n}+3u_{i^*, n}) + \sum_{j\in V, j\neq i}(u_{j,n}-u_{i^*, n}).
	\end{align*}
	Since both sequences
	$$\left(\sum_{j\in V, j\neq i}(u_{j,n}-u_{i^*, n})\right)_{n\in\N}\qquad\text{and}\qquad(u_{i, n}+ 3u_{i^*, n} )_{n\in\N}$$
	are bounded from below, and the sequence~$(-2u_{i^*, n})_{n\in\N}$ diverges properly to~$+\infty$, 
	$$\lim_{n\to\infty}\E(u_{n}) = +\infty.$$	
\end{proof}

\begin{acknowledgements}
	I want to thank prof. Boris Springborn for his support during the writing of this article. 
\end{acknowledgements}
\bibliographystyle{plain}
\bibliography{refs2}

%
%


\end{document}